\documentclass[a4paper,12pt]{article}%
\usepackage{amsmath}
\usepackage{amsfonts}
\usepackage{amssymb}
\usepackage{graphicx}%
\newtheorem{theorem}{Theorem}

\newtheorem{corollary}[theorem]{Corollary}

\newtheorem{lemma}[theorem]{Lemma}

\newtheorem{remark}[theorem]{Remark}

\newenvironment{proof}[1][Proof]{\noindent\textbf{#1.} }{\ \rule{0.5em}{0.5em}}
\begin{document}

\title{Trigonometric approximation and a general form of the Erd\H{o}s Tur\'{a}n inequality}
\author{Leonardo Colzani
\and Giacomo Gigante
\and Giancarlo Travaglini}
\date{}
\maketitle

\begin{abstract}
There exists a positive function $\psi(t)${\ on }$t\geq0${, with fast decay at
infinity, such that for every measurable set }$\Omega${\ in the Euclidean
space\ and }$R>0${, there exist entire functions }$A\left(  x\right)  ${\ and
}$B\left(  x\right)  ${\ of exponential type }$R${, satisfying\ }$A(x)\leq
\chi_{\Omega}(x)\leq B(x)${\ and }$\left\vert B(x)-A(x)\right\vert
\leqslant\psi\left(  R\operatorname*{dist}\left(  x,\partial\Omega\right)
\right)  $. This leads to Erd\H{o}s Tur\'{a}n estimates for discrepancy of
point set distributions in the multi dimensional torus. Analogous results hold
for approximations by eigenfunctions of differential operators and discrepancy
on compact manifolds.

\end{abstract}

An infinite sequence of points $\left\{  x_{j}\right\}  _{j=1}^{+\infty}$ is
uniformly distributed in the interval $0\leq x\leq1$ if, for every $0\leq
a<b\leq1$,
\[
\lim_{m\rightarrow+\infty}\left\{  m^{-1}%
%TCIMACRO{\dsum _{j=1}^{m}}%
%BeginExpansion
{\displaystyle\sum_{j=1}^{m}}
%EndExpansion
\chi_{\lbrack a,b]}\left(  x_{j}\right)  \right\}  =b-a.
\]

A quantitative measure of the irregularity of distribution of the points
$\left\{  x_{j}\right\}  _{j=1}^{m}$ is given by the discrepancy
\[
\sup_{0\leq a<b\leq1}\left|  (b-a)-m^{-1}%
%TCIMACRO{\dsum _{j=1}^{m}}%
%BeginExpansion
{\displaystyle\sum_{j=1}^{m}}
%EndExpansion
\chi_{\lbrack a,b]}\left(  x_{j}\right)  \right|  .
\]

Upper bounds for the discrepancy are useful because they lead to upper bounds
for the approximation of integrals by Riemann sums. A well known criterion due
to H.Weyl states that a sequence is uniformly distributed if and only if for
every $k\neq0$,
\[
\lim_{m\rightarrow+\infty}\left\{  m^{-1}%
%TCIMACRO{\dsum _{j=1}^{m}}%
%BeginExpansion
{\displaystyle\sum_{j=1}^{m}}
%EndExpansion
\exp\left(  2\pi ikx_{j}\right)  \right\}  =0.
\]

An upper bound for discrepancy is given by the classical inequality of
P.Erd\H{o}s and P.Tur\'{a}n:
\begin{gather*}
\sup_{0\leq a<b\leq1}\left\vert (b-a)-m^{-1}%
%TCIMACRO{\dsum _{j=1}^{m}}%
%BeginExpansion
{\displaystyle\sum_{j=1}^{m}}
%EndExpansion
\chi_{\lbrack a,b]}\left(  x_{j}\right)  \right\vert \\
\leq\inf_{n=1,2,...}\left\{  \dfrac{6}{n+1}+\dfrac{4}{\pi}%
%TCIMACRO{\dsum _{k=1}^{n}}%
%BeginExpansion
{\displaystyle\sum_{k=1}^{n}}
%EndExpansion
\left(  \dfrac{1}{k}-\dfrac{1}{n+1}\right)  \left\vert m^{-1}%
%TCIMACRO{\dsum _{j=1}^{m}}%
%BeginExpansion
{\displaystyle\sum_{j=1}^{m}}
%EndExpansion
\exp\left(  2\pi ikx_{j}\right)  \right\vert \right\}  .
\end{gather*}

A proof of the above inequality relies on approximations from above and below
of the characteristic functions $\chi_{\lbrack a,b]}\left(  x\right)  $ by
trigonometric polynomials $P(x)=\sum_{k=-n}^{+n}\widehat{P}(k)\exp\left(  2\pi
ikx\right)  $, so that
\begin{gather*}
(b-a)-m^{-1}%
%TCIMACRO{\dsum _{j=1}^{m}}%
%BeginExpansion
{\displaystyle\sum_{j=1}^{m}}
%EndExpansion
\chi_{\lbrack a,b]}\left(  x_{j}\right)  \approx(b-a)-m^{-1}%
%TCIMACRO{\dsum _{j=1}^{m}}%
%BeginExpansion
{\displaystyle\sum_{j=1}^{m}}
%EndExpansion
P\left(  x_{j}\right) \\
=(b-a)-\widehat{P}(0)-%
%TCIMACRO{\dsum _{1\leq\left\vert k\right\vert \leq n}}%
%BeginExpansion
{\displaystyle\sum_{1\leq\left\vert k\right\vert \leq n}}
%EndExpansion
\left(  m^{-1}%
%TCIMACRO{\dsum _{j=1}^{m}}%
%BeginExpansion
{\displaystyle\sum_{j=1}^{m}}
%EndExpansion
\exp\left(  2\pi ikx_{j}\right)  \right)  \widehat{P}(k).
\end{gather*}

See \cite{Beck-Chen,Erdos-Turan,Kuipers-Niederreiter}. A deep study of the
approximation of the characteristic function of an interval by trigonometric
polynomials has been done by A.Beurling and A.Selberg, who proved that for
every $0\leq a\leq b\leq1$ and $n=0,1,2,...$, there exist trigonometric
polynomials $P_{\pm}\left(  x\right)  $ of degree $n$ with
\begin{gather*}
P_{-}\left(  x\right)  \leq\chi_{\lbrack a,b]}\left(  x\right)  \leq
P_{+}\left(  x\right)  ,\\%
%TCIMACRO{\dint _{0}^{1}}%
%BeginExpansion
{\displaystyle\int_{0}^{1}}
%EndExpansion
\left\vert P_{\pm}\left(  x\right)  -\chi_{\lbrack a,b]}\left(  x\right)
\right\vert dx=1/(n+1).
\end{gather*}

See e.g. \cite{Montgomery}. Similar extremal problems have been considered in
\cite{Holt-Vaaler}, with precise estimates on the approximation, in
$-\infty<x<+\infty$ with measure $\left|  x\right|  ^{2\nu+1}dx$, of the
function $\operatorname{sgn}\left(  x\right)  $ by functions of finite
exponential type. A radialization of these functions then yields an analog of
Selberg polynomials for approximation of characteristic functions of
multi-dimensional balls and this has been applied to Erd\H{o}s Tur\'{a}n
estimates of discrepancy. See \cite{Harman,Holt,Holt-Vaaler}, and also
\cite{Cochrane} for a generalization to boxes.

Here we look for a geometric analog of Erd\H{o}s and Tur\'{a}n results in a
very general framework, where the intervals in the torus are replaced by
arbitrary measurable sets on a manifold, and trigonometric polynomials are
replaced by finite linear combinations of eigenfunctions of the Laplace
Beltrami operator. In spite of this generality, the techniques are rather
simple, and they may be of some interest even in the classical Euclidean
setting. Moreover, the results are optimal, up to the constants involved. In
fact we shall try to pay special attention to these constants, which are quite
explicit although not optimal. In this sense we acknowledge that we do not
match the beauty of Beurling and Selberg results, which stems precisely in
their extremal properties.

The plan of the paper is as follows. The first section is devoted to
approximations from above and below of characteristic functions by entire
functions of exponential type or, in the periodic setting, approximations by
trigonometric polynomials. The main result of this section is the following:

\begin{theorem}
{\ There exists a positive function }$\psi(t)${\ on }$t\geq0${\ with fast
decay at infinity, }$\psi(t)\leq c(\alpha)\left(  1+t\right)  ^{-\alpha}%
${\ for every }$\alpha${, such that for every measurable set }$\Omega${\ in
the Euclidean space }$\mathbb{R}^{d}${\ and }$R>0${, there exist entire
functions }$A(x)${\ and }$B(x)${\ of exponential type }$R${,\ satisfying}
\[
A(x)\leq\chi_{\Omega}(x)\leq B(x),\;\;\;\left|  B(x)-A(x)\right|
\leqslant\psi\left(  R\operatorname*{dist}\left(  x,\partial\Omega\right)
\right)  .
\]

\end{theorem}

Roughly speaking, the approximation $\left\vert B(x)-A(x)\right\vert $ is
essentially $1$ at points with distances $1/R$ from the boundary of $\Omega$,
while $\left\vert B(x)-A(x)\right\vert $ is essentially $0$ at larger
distances. We would like to emphasize that the function $\psi(t)$ in the
theorem is independent of the set $\Omega$ and that there are no regularity
assumptions on this set. If the set is regular, then there are few points at
small distance from the boundary and the approximation is bad only on a small
set. If the set is fractal, then there are many points at small distance from
the boundary and the approximation is bad on a large set. In particular, this
approximation is related to the Minkowski content of the boundary. See
\cite{Falconer}.

In the second section these approximations are applied to Erd\H{o}s Tur\'{a}n
estimates of irregularities of point distribution on a torus. In particular,
following \cite{Kuipers-Niederreiter} and \cite{Schmidt}, we obtain explicit
estimates for the discrepancy of sequences in lattices and in arithmetic
progression, which improve and extend some results already in the literature.
Here are some examples of the results in the second section (see,
respectively, Corollaries \ref{corlattice}, \ref{10}, and
\ref{corgoodlatticepoint} below, where these results are proved).

\begin{theorem}
\label{lattice}{\ If }$0\leq\alpha\leq1$,{ }$\mu${ is the Lebesgue measure,
}$\Omega${ is a measurable set in the torus }$\mathbb{T}^{d}=\mathbb{R}%
^{d}/\mathbb{Z}^{d}${, let }%
\[
M\left(  \alpha,\Omega\right)  =\sup_{t>0}t^{-\alpha}\mu\left(  \left\{
\operatorname*{dist}\left(  x,\partial\Omega\right)  <t\right\}  \right)
\]
{(this quantity is related to the Minkowski content of }$\partial\Omega${).
Also, if }$m^{1/d}${ is a positive integer, let }$L(m)=m^{-1/d}\mathbb{Z}^{d}%
${ be the lattice points }$m^{-1/d}\left(  g_{1},...,g_{d}\right)  ${ with
}$g_{j}=0,1,...,m^{1/d}-1${. Then there exists a constant }$c${ such that for
every such }$m\geq1$,{ }%
\[
\sup_{M\left(  \alpha,\Omega\right)  <\gamma}\left\vert \mu(\Omega)-m^{-1}%
%TCIMACRO{\dsum _{x\in L(m)}}%
%BeginExpansion
{\displaystyle\sum_{x\in L(m)}}
%EndExpansion
\chi_{\Omega}\left(  x\right)  \right\vert \leq c\gamma m^{-\alpha/d}.
\]

\end{theorem}

\begin{theorem}
\label{ae}{\ Given }$\varepsilon>0${, for almost every }$x${\ in }%
$\mathbb{T}^{d}${ there exists a constant }$c${ such that for every }$m>1${, }%
\[
\sup_{M\left(  \alpha,\Omega\right)  <\gamma}\left|  \mu(\Omega)-m^{-1}%
%TCIMACRO{\dsum _{j=1}^{m}}%
%BeginExpansion
{\displaystyle\sum_{j=1}^{m}}
%EndExpansion
\chi_{\Omega}(jx)\right|  \leq c\gamma m^{-\alpha/d}\log^{\alpha\left(
d+1+\varepsilon\right)  /d}\left(  m\right)  .
\]

\end{theorem}

\begin{theorem}
\label{polys} {\ Let }$X${ be a finite collection of hyperspaces of
}$\mathbb{R}^{d}${ and let }$P(X)${ be the collection of all convex polyhedra
in the torus }$\mathbb{T}^{d}${ with diameter smaller than }$1${ and facets
parallel to elements of }$X${. Then there exists a constant }$c${ such that,
given a prime number }$m${, there exists a lattice point }$g=\left(
g_{1},...,g_{d}\right)  ${ in }$\mathbb{Z}^{d}${ with }$1\leq g_{j}\leq m-1${,
such that }%
\[
\sup_{\Omega\in P(X)}\left\vert \mu(\Omega)-m^{-1}%
%TCIMACRO{\dsum _{j=1}^{m}}%
%BeginExpansion
{\displaystyle\sum_{j=1}^{m}}
%EndExpansion
\chi_{\Omega}\left(  jm^{-1}g\right)  \right\vert \leq cm^{-1}\log^{d}(m).
\]

\end{theorem}

In particular, up to a logarithmic transgression, the discrepancy of a random
arithmetic progression $\left\{  jx\right\}  _{j=1}^{m}$ is comparable to the
one of the lattice $L(m)$ with the same number of points. On the other hand,
while the first two theorems applied to polyhedra with $\alpha=1$ give the
bound $m^{-1/d}$, the third theorem gives the better bound $m^{-1}$. This
improves and extends a two dimensional result in \cite[Theorem 4D]{Beck},
where it is proved that the discrepancy of $m$ points with respect to polygons
is dominated by $m^{-1}\log^{4+\varepsilon}\left(  m\right)  $.

In the third section the results obtained on the Euclidean settings are
extended to compact Riemannian manifolds. We first consider approximations by
linear combinations of eigenfunctions of the Laplace Beltrami operator on a
compact Riemannian manifold. Then we state an analog of Erd\H{o}s Tur\'{a}n
estimates for irregularities of point distribution on manifolds and, inspired
by \cite{Lubotzky-Phillips-Sarnak}, when the manifold is a compact Lie group
or homogeneous space we consider point distributions generated by the action
of a free group. In particular, using the Ramanujan bounds for eigenvalues of
Hecke operators obtained in the above quoted paper, we prove the following
result (see Corollary \ref{sphere} below).

\begin{theorem}
\label{sphereintro} {If }$\mathcal{M}=SO(3)/SO(2)${\ is the two dimensional
sphere, if }$\mathcal{H}${\ is the free group generated by rotations of angles
}$\arccos(-3/5)${\ around orthogonal axes, then there exists a constant }$c$
{such that,} {if} $k${ is an integer and }$\left\{  \sigma_{j}\right\}
_{j=1}^{m}${ is an ordering of the elements in }$\mathcal{H}${ with length at
most }$k$,{ then for every }$x$,
\[
\left\vert \mu(\Omega)-m^{-1}%
%TCIMACRO{\dsum _{j=1}^{m}}%
%BeginExpansion
{\displaystyle\sum_{j=1}^{m}}
%EndExpansion
\chi_{\Omega}(\sigma_{j}x)\right\vert \leq cM\left(  \delta,\Omega\right)
m^{-\delta/(2+\delta)}\log^{2\delta/(2+\delta)}(m).
\]

\end{theorem}

This result has been proved in \cite{Lubotzky-Phillips-Sarnak} in the case of
spherical caps, with a proof that relies on explicit estimates of Fourier
coefficients. On the contrary, our result applies to more general domains.

The authors wish to thank Luca Brandolini for several useful discussions on
the subject of the paper.

\section{Approximation by entire functions}

The main result in this section is the following.

\begin{theorem}
\label{1}{\ There exists a positive function }$\psi(t)${\ on }$t\geq0${\ with
fast decay at infinity, }$\psi(t)\leq c(\alpha)\left(  1+t\right)  ^{-\alpha}%
${\ for every }$\alpha${, such that for every measurable set }$\Omega${\ in
the Euclidean space }$\mathbb{R}^{d}${\ and }$R>0${, there exist entire
functions }$A(x)${\ and }$B(x)${\ of exponential type }$R${,\ satisfying}
\[
A(x)\leq\chi_{\Omega}(x)\leq B(x),\;\;\;\left|  B(x)-A(x)\right|
\leqslant\psi\left(  R\operatorname*{dist}\left(  x,\partial\Omega\right)
\right)  .
\]

\end{theorem}

\begin{proof}
By a scaling argument, the statement for the set $\Omega$ at the point $x$
with functions of exponential type $R$, is equivalent to the statement for the
set $R\Omega$ at the point $Rx$ with functions of exponential type $1$. Hence,
it suffices to prove the theorem when $R=1$. Let $m(\xi)$ be a smooth radial
function on $\mathbb{R}^{d}$ with $m(\xi)=0$ if $\left|  \xi\right|  \geq1/2$
and $\int_{\mathbb{R}^{d}}m\left(  \xi\right)  ^{2}d\xi=1$. Then the
convolution $m\ast m\left(  \xi\right)  $ is a smooth radial function with
$m\ast m\left(  0\right)  =1$ and $m(\xi)=0$ if $\left|  \xi\right|  \geq1$.
Define
\[
K(x)=%
%TCIMACRO{\dint _{\mathbb{R}^{d}}}%
%BeginExpansion
{\displaystyle\int_{\mathbb{R}^{d}}}
%EndExpansion
\left(  1+\left|  \xi\right|  ^{2}\right)  ^{-(d+1)/2}m\ast m\left(
\xi\right)  \exp\left(  2\pi i\xi\cdot x\right)  d\xi.
\]
This cumbersome definition will be clarified in a series of steps.

\noindent\textit{Claim.} {The kernel }$K(x)${\ is an entire function of finite
exponential type, it is positive with mean 1 on }$\mathbb{R}^{d}${\ and all
its derivatives have fast decay at infinity, }$\left|  \partial^{\beta
}K(x)/\partial x^{\beta}\right|  \leq c\left(  1+\left|  x\right|  \right)
^{-\alpha}${\ for every }$\alpha${\ and }$\beta${.}

All of this follows from the corresponding properties of the Fourier transform
$\widehat{K}(\xi)=\left(  1+\left|  \xi\right|  ^{2}\right)  ^{-(d+1)/2}m\ast
m\left(  \xi\right)  $. Since this Fourier transform is smooth with compact
support and it is 1 at the origin, the kernel has mean 1 and all its
derivatives have fast decay at infinity. Since the kernel is the convolution
of the Fourier transform of $\left(  1+\left|  \xi\right|  ^{2}\right)
^{-(d+1)/2}$, which is up to a constant $\exp\left(  -2\pi\left|  x\right|
\right)  $, and the square of the Fourier transform of $m\left(  \xi\right)
$, the kernel is positive. Finally, since $\widehat{K}(\xi)=0$ if $\left|
\xi\right|  \geq1$, by the Paley Wiener theorem the kernel is an entire
function of finite exponential type,
\begin{gather*}
K(x+iy)=%
%TCIMACRO{\dint _{\mathbb{R}^{d}}}%
%BeginExpansion
{\displaystyle\int_{\mathbb{R}^{d}}}
%EndExpansion
\left(  1+\left|  \xi\right|  ^{2}\right)  ^{-(d+1)/2}m\ast m\left(
\xi\right)  \exp\left(  2\pi i\xi\cdot\left(  x+iy\right)  \right)  d\xi,\\
\left|  K(x+iy)\right|  \leq\left(
%TCIMACRO{\dint _{\mathbb{R}^{d}}}%
%BeginExpansion
{\displaystyle\int_{\mathbb{R}^{d}}}
%EndExpansion
\left(  1+\left|  \xi\right|  ^{2}\right)  ^{-(d+1)/2}\left|  m\ast m\left(
\xi\right)  \right|  d\xi\right)  \exp\left(  2\pi\left|  y\right|  \right)  .
\end{gather*}

\noindent\textit{Claim.}{\ Let}
\[
I(t)=%
%TCIMACRO{\dint _{\left\{  \left|  x\right|  \geq t\right\}  }}%
%BeginExpansion
{\displaystyle\int_{\left\{  \left|  x\right|  \geq t\right\}  }}
%EndExpansion
K(x)dx.
\]
{Then, for every }$t\geq0${, }$I(t+1)\geq\exp\left(  -2\pi\right)  I(t)${.}

Since $K(x)$ is the convolution of $c\exp\left(  -2\pi\left|  x\right|
\right)  $ and the positive function $\left|  \widehat{m}(x)\right|  ^{2}$,
one has
\begin{gather*}
K(x+y)=%
%TCIMACRO{\dint _{\mathbb{R}^{d}}}%
%BeginExpansion
{\displaystyle\int_{\mathbb{R}^{d}}}
%EndExpansion
c\exp\left(  -2\pi\left|  x+y-z\right|  \right)  \left|  \widehat
{m}(z)\right|  ^{2}dz\\
\geq\exp\left(  -2\pi\left|  y\right|  \right)
%TCIMACRO{\dint _{\mathbb{R}^{d}}}%
%BeginExpansion
{\displaystyle\int_{\mathbb{R}^{d}}}
%EndExpansion
c\exp\left(  -2\pi\left|  x-z\right|  \right)  \left|  \widehat{m}(z)\right|
^{2}dz=\exp\left(  -2\pi\left|  y\right|  \right)  K(x).
\end{gather*}
Hence
\begin{gather*}
I(t+1)=%
%TCIMACRO{\dint _{\left\{  \left|  \vartheta\right|  =1\right\}  }}%
%BeginExpansion
{\displaystyle\int_{\left\{  \left|  \vartheta\right|  =1\right\}  }}
%EndExpansion%
%TCIMACRO{\dint _{t+1}^{+\infty}}%
%BeginExpansion
{\displaystyle\int_{t+1}^{+\infty}}
%EndExpansion
K(\rho\vartheta)\rho^{d-1}d\rho d\vartheta\\
=%
%TCIMACRO{\dint _{\left\{  \left|  \vartheta\right|  =1\right\}  }}%
%BeginExpansion
{\displaystyle\int_{\left\{  \left|  \vartheta\right|  =1\right\}  }}
%EndExpansion%
%TCIMACRO{\dint _{t}^{+\infty}}%
%BeginExpansion
{\displaystyle\int_{t}^{+\infty}}
%EndExpansion
K\left(  \left(  \tau+1\right)  \vartheta\right)  )\left(  \tau+1\right)
^{d-1}d\tau d\vartheta\\
\geq\exp\left(  -2\pi\right)
%TCIMACRO{\dint _{\left\{  \left|  \vartheta\right|  =1\right\}  }}%
%BeginExpansion
{\displaystyle\int_{\left\{  \left|  \vartheta\right|  =1\right\}  }}
%EndExpansion%
%TCIMACRO{\dint _{t}^{+\infty}}%
%BeginExpansion
{\displaystyle\int_{t}^{+\infty}}
%EndExpansion
K\left(  \tau\vartheta\right)  )\tau^{d-1}d\tau d\vartheta=\exp\left(
-2\pi\right)  I(t).
\end{gather*}

\noindent\textit{Claim.} {Define }
\[
G\left(  x\right)  =I\left(  \operatorname*{dist}\left(  x,\partial
\Omega\right)  \right)  .
\]
{Then, for every }$x${, }
\[
\left|  \chi_{\Omega}\left(  x\right)  -K\ast\chi_{\Omega}\left(  x\right)
\right|  \leq G\left(  x\right)  .
\]

Since $K(y)$ is positive with mean $1$,
\begin{gather*}
\left\vert \chi_{\Omega}\left(  x\right)  -K\ast\chi_{\Omega}\left(  x\right)
\right\vert =\left\vert
%TCIMACRO{\dint _{\mathbb{R}^{d}}}%
%BeginExpansion
{\displaystyle\int_{\mathbb{R}^{d}}}
%EndExpansion
K\left(  y\right)  \left(  \chi_{\Omega}\left(  x\right)  -\chi_{\Omega
}\left(  x-y\right)  \right)  dy\right\vert \\
\leq%
%TCIMACRO{\dint _{\left\{  \left\vert y\right\vert \geq\operatorname*{dist}%
%\left(  x,\partial\Omega\right)  \right\}  }}%
%BeginExpansion
{\displaystyle\int_{\left\{  \left\vert y\right\vert \geq\operatorname*{dist}%
\left(  x,\partial\Omega\right)  \right\}  }}
%EndExpansion
K(y)dy=G\left(  x\right)  .
\end{gather*}

\noindent\textit{Claim.}{\ Define }
\[
F\left(  x\right)  =I\left(  \operatorname*{dist}\left(  x,\partial
\Omega\right)  /2\right)  .
\]
{Then, for every }$x${, }
\[
K\ast G\left(  x\right)  \leq2F(x).
\]

Since $\operatorname*{dist}\left(  x-y,\partial\Omega\right)  \geq
\operatorname*{dist}\left(  x,\partial\Omega\right)  -\left|  y\right|  $, it
follows that
\[
\left\{  \left|  z\right|  \geq\operatorname*{dist}\left(  x-y,\partial
\Omega\right)  \right\}  \subseteq\left\{  \left|  z\right|  \geq
\operatorname*{dist}\left(  x,\partial\Omega\right)  -\left|  y\right|
\right\}  .
\]
Hence
\begin{gather*}
K\ast G\left(  x\right)  =%
%TCIMACRO{\dint _{\mathbb{R}^{d}}}%
%BeginExpansion
{\displaystyle\int_{\mathbb{R}^{d}}}
%EndExpansion%
%TCIMACRO{\dint _{\left\{  \left|  z\right|  \geq\operatorname*{dist}\left(
%x-y,\partial\Omega\right)  \right\}  }}%
%BeginExpansion
{\displaystyle\int_{\left\{  \left|  z\right|  \geq\operatorname*{dist}\left(
x-y,\partial\Omega\right)  \right\}  }}
%EndExpansion
K\left(  y\right)  K\left(  z\right)  dzdy\\
\leq%
%TCIMACRO{\dint _{\left\{  \left|  y\right|  \leq\operatorname*{dist}\left(
%x,\partial\Omega\right)  /2\right\}  }}%
%BeginExpansion
{\displaystyle\int_{\left\{  \left|  y\right|  \leq\operatorname*{dist}\left(
x,\partial\Omega\right)  /2\right\}  }}
%EndExpansion%
%TCIMACRO{\dint _{\left\{  \left|  z\right|  \geq\operatorname*{dist}\left(
%x,\partial\Omega\right)  /2\right\}  }}%
%BeginExpansion
{\displaystyle\int_{\left\{  \left|  z\right|  \geq\operatorname*{dist}\left(
x,\partial\Omega\right)  /2\right\}  }}
%EndExpansion
K\left(  y\right)  K\left(  z\right)  dzdy\\
+%
%TCIMACRO{\dint _{\left\{  \left|  y\right|  \geq\operatorname*{dist}\left(
%x,\partial\Omega\right)  /2\right\}  }}%
%BeginExpansion
{\displaystyle\int_{\left\{  \left|  y\right|  \geq\operatorname*{dist}\left(
x,\partial\Omega\right)  /2\right\}  }}
%EndExpansion%
%TCIMACRO{\dint _{\mathbb{R}^{d}}}%
%BeginExpansion
{\displaystyle\int_{\mathbb{R}^{d}}}
%EndExpansion
K\left(  y\right)  K\left(  z\right)  dzdy\\
\leq2%
%TCIMACRO{\dint _{\left\{  \left|  y\right|  \geq\operatorname*{dist}\left(
%x,\partial\Omega\right)  /2\right\}  }}%
%BeginExpansion
{\displaystyle\int_{\left\{  \left|  y\right|  \geq\operatorname*{dist}\left(
x,\partial\Omega\right)  /2\right\}  }}
%EndExpansion
K\left(  y\right)  dy.
\end{gather*}

\noindent\textit{Claim.}{\ Let }
\[
\gamma^{-1}=\exp\left(  -2\pi\right)
%TCIMACRO{\dint _{\left\{  \left|  y\right|  \leq1\right\}  }}%
%BeginExpansion
{\displaystyle\int_{\left\{  \left|  y\right|  \leq1\right\}  }}
%EndExpansion
K\left(  y\right)  dy.
\]
{Then for every }$x${,}
\[
K\ast G\left(  x\right)  \geq\gamma^{-1}G(x).
\]

Since $\operatorname*{dist}\left(  x-y,\partial\Omega\right)  \leq
\operatorname*{dist}\left(  x,\partial\Omega\right)  +\left|  y\right|  $, it
follows that
\[
\left\{  \left|  z\right|  \geq\operatorname*{dist}\left(  x-y,\partial
\Omega\right)  \right\}  \supseteq\left\{  \left|  z\right|  \geq
\operatorname*{dist}\left(  x,\partial\Omega\right)  +\left|  y\right|
\right\}  .
\]
Hence
\begin{gather*}
K\ast G\left(  x\right)  =%
%TCIMACRO{\dint _{\mathbb{R}^{d}}}%
%BeginExpansion
{\displaystyle\int_{\mathbb{R}^{d}}}
%EndExpansion%
%TCIMACRO{\dint _{\left\{  \left|  z\right|  \geq\operatorname*{dist}\left(
%x-y,\partial\Omega\right)  \right\}  }}%
%BeginExpansion
{\displaystyle\int_{\left\{  \left|  z\right|  \geq\operatorname*{dist}\left(
x-y,\partial\Omega\right)  \right\}  }}
%EndExpansion
K\left(  y\right)  K\left(  z\right)  dzdy\\
\geq%
%TCIMACRO{\dint _{\left\{  \left|  y\right|  \leq1\right\}  }}%
%BeginExpansion
{\displaystyle\int_{\left\{  \left|  y\right|  \leq1\right\}  }}
%EndExpansion%
%TCIMACRO{\dint _{\left\{  \left|  z\right|  \geq\operatorname*{dist}\left(
%x,\partial\Omega\right)  +1\right\}  }}%
%BeginExpansion
{\displaystyle\int_{\left\{  \left|  z\right|  \geq\operatorname*{dist}\left(
x,\partial\Omega\right)  +1\right\}  }}
%EndExpansion
K\left(  y\right)  K\left(  z\right)  dzdy\\
\geq\gamma^{-1}I\left(  \operatorname*{dist}\left(  x,\partial\Omega\right)
\right)  .
\end{gather*}

To conclude the proof of the theorem, define $H(x)=\gamma G\left(  x\right)  $
and
\begin{align*}
A(x)  &  =K\ast\chi_{\Omega}\left(  x\right)  -K\ast H\left(  x\right)  ,\\
B(x)  &  =K\ast\chi_{\Omega}\left(  x\right)  +K\ast H\left(  x\right)  .
\end{align*}

Since the kernel $K(x)$ is an entire function of finite exponential type, then
also the convolutions with this kernel are entire functions of finite
exponential type. In particular, both $A(x)$ and $B(x)$ are entire functions
of exponential type not larger than the one of $K(x)$. Moreover, by the above
claims,
\begin{align*}
\chi_{\Omega}\left(  x\right)  -A(x)  &  =K\ast H\left(  x\right)  -\left(
K\ast\chi_{\Omega}\left(  x\right)  -\chi_{\Omega}\left(  x\right)  \right)
\geq0,\\
B(x)-\chi_{\Omega}\left(  x\right)   &  =K\ast H\left(  x\right)  -\left(
\chi_{\Omega}\left(  x\right)  -K\ast\chi_{\Omega}\left(  x\right)  \right)
\geq0.
\end{align*}
Finally,
\[
B(x)-A\left(  x\right)  =2K\ast H\left(  x\right)  \leq4\gamma F\left(
x\right)  =4\gamma I\left(  \operatorname*{dist}\left(  x,\partial
\Omega\right)  /2\right)  .
\]
Hence, the theorem follows with $\psi(t)=4\gamma I(t/2)$.
\end{proof}

It follows from the proof of the theorem that for periodic sets with respect
to the integer lattice $\mathbb{Z}^{d}$, the above approximating entire
functions are periodic too, hence they are trigonometric polynomials.

\begin{corollary}
\label{2}{\ There exists a positive function }$\psi(t)${\ with fast decay at
infinity,\ such that for every measurable set }$\Omega${\ in the torus
}$\mathbb{T}^{d}=\mathbb{R}^{d}/\mathbb{Z}^{d}${\ and }$R=0,1,2,...${, there
exist\ trigonometric polynomials }$A(x)${\ and }$B(x)${\ of degree }%
$R${\ with}
\[
A(x)\leq\chi_{\Omega}(x)\leq B(x),\;\;\;\left|  B(x)-A(x)\right|
\leqslant\psi\left(  R\operatorname*{dist}\left(  x,\partial\Omega\right)
\right)  .
\]

\end{corollary}

\begin{proof}
The corollary follows immediately from the theorem. However, in order to
clarify what follows, we write down explicitly the Fourier expansions of the
trigonometric approximations. Let $\Omega=\Omega+\mathbb{Z}^{d}$ be a
$\mathbb{Z}^{d}$-periodic set in $\mathbb{R}^{d}$. As in the proof of the
theorem, for every $R>0$, let
\begin{gather*}
K_{R}\left(  x\right)  =\sum_{k\in\mathbb{Z}^{d}}R^{d}K\left(  R\left(
x+k\right)  \right)  =\sum_{k\in\mathbb{Z}^{d}}\widehat{K}\left(  k/R\right)
\exp\left(  2\pi ik\cdot x\right)  ,\\
H_{R}\left(  x\right)  =\exp\left(  2\pi\right)  \left(
%TCIMACRO{\dint _{\left\{  \left|  y\right|  \leq1\right\}  }}%
%BeginExpansion
{\displaystyle\int_{\left\{  \left|  y\right|  \leq1\right\}  }}
%EndExpansion
K\left(  y\right)  dy\right)  ^{-1}\int_{\left\{  \left|  y\right|  \geq
R\operatorname*{dist}\left(  x,\partial\Omega+\mathbb{Z}^{d}\right)  \right\}
}K\left(  y\right)  dy.
\end{gather*}
Then,
\begin{gather*}
A\left(  x\right)  ,B\left(  x\right)  =\int_{\mathbb{T}^{d}}K_{R}\left(
y\right)  \left(  \chi_{\Omega}\left(  x-y\right)  \mp H_{R}\left(
x-y\right)  \right)  dy\\
=\sum_{k\in\mathbb{Z}^{d}}\widehat{K}\left(  k/R\right)  \left(  \widehat
{\chi}_{\Omega}\left(  k\right)  \mp\widehat{H}_{R}\left(  k\right)  \right)
\exp\left(  2\pi ik\cdot x\right)  .
\end{gather*}

\end{proof}

\begin{remark}
As we have said, in the above theorem the approximation $\left|
B(x)-A(x)\right|  $ is controlled by $1$ at points with distances $1/R$ from
the boundary of $\Omega$, while $\left|  B(x)-A(x)\right|  $ is essentially
$0$ at larger distances. It follows from the inequality of S.Bernstein between
the maxima of an entire function and its derivatives that this approximation
is essentially optimal. Indeed, if $C(z)$ is an entire function of exponential
type $1$, then
\[
\left|  C(x)-C(y)\right|  \leq\sup_{z\in\mathbb{R}^{d}}\left|  \nabla
C(z)\right|  \left|  x-y\right|  \leq2\pi\sup_{z\in\mathbb{R}^{d}}\left|
C(z)\right|  \left|  x-y\right|  .
\]
Hence, if $A(x)\leq\chi_{\Omega}(x)\leq B(x)$ are entire functions of
exponential type 1, if $x$ is in $\Omega$ and $y$ is outside $\Omega$ with
$\left|  x-y\right|  \leq\left(  4\pi\sup_{z\in\mathbb{R}^{d}}\left\{  \left|
A(z)\right|  ,\left|  B(z)\right|  \right\}  \right)  ^{-1}$, then
$B(x)-A(y)\geq1$ and
\begin{align*}
B(x)-A(x)  &  =\left(  B(x)-A(y)\right)  -\left(  A(x)-A(y)\right)
\geq1-1/2,\\
B(y)-A(y)  &  =\left(  B(x)-A(y)\right)  -\left(  B(x)-B(y)\right)  \geq1-1/2.
\end{align*}

\end{remark}

\section{Discrepancy on the torus}

As advertised in the introduction, the approximation results in the previous
section have simple and straightforward applications to multi dimensional
versions of the classical Erd\H{o}s Tur\'{a}n inequality for discrepancy of
point distribution. In the sequel, for simplicity, the sets $\Omega$ in the
torus $\mathbb{T}^{d}=\mathbb{R}^{d}/\mathbb{Z}^{d}$ will be periodic sets in
$\mathbb{R}^{d}$ of the form $\Omega=\Omega^{\ast}+\mathbb{Z}^{d}$, with
$\Omega^{\ast}$ in $\mathbb{R}^{d}$ with the property that
$\operatorname*{dist}\left(  \Omega^{\ast},\Omega^{\ast}+k\right)
>\varepsilon>0$ for every $k\in\mathbb{Z}^{d}-\left\{  0\right\}  $.
%In fact,
%any periodic set is a disjoint union of at most $2^{d}$ sets with this
%property.
With this identification, the distance of a point $x$ from $\partial\Omega$ in
$\mathbb{T}^{d}$ is the distance of $x$ from $\partial\Omega^{\ast}%
+\mathbb{Z}^{d}$ in $\mathbb{R}^{d}$.

\begin{theorem}
\label{4}{\ If }$\left\{  x_{j}\right\}  _{j=1}^{m}${ is a sequence of points
in the torus, if }$\Omega${ is a measurable set with measure }$\mu(\Omega)${,
and if }$H_{R}(x)=4^{-1}\psi\left(  2R\operatorname*{dist}\left(
x,\partial\Omega\right)  \right)  ${ with }$R>0${ and }$\psi\left(  t\right)
${ with fast decay at infinity, as in the proof of Corollary \ref{2}, then }%
\begin{gather*}
\left\vert \mu(\Omega)-m^{-1}%
%TCIMACRO{\dsum _{j=1}^{m}}%
%BeginExpansion
{\displaystyle\sum_{j=1}^{m}}
%EndExpansion
\chi_{\Omega}(x_{j})\right\vert \leq\\
\left\vert \widehat{H}_{R}(0)\right\vert +\sum_{0<\left\vert k\right\vert
<R}\left(  \left\vert \widehat{\chi}_{\Omega}(k)\right\vert +\left\vert
\widehat{H}_{R}(k)\right\vert \right)  \left\vert m^{-1}%
%TCIMACRO{\dsum _{j=1}^{m}}%
%BeginExpansion
{\displaystyle\sum_{j=1}^{m}}
%EndExpansion
\exp\left(  2\pi ik\cdot x_{j}\right)  \right\vert .
\end{gather*}

\end{theorem}

\begin{proof}
If $A(x)\leq\chi_{\Omega}(x)\leq B(x)$ are defined as in the proof of
Corollary \ref{2}, then
\[
A(x),B(x)=\sum_{k\in\mathbb{Z}^{d}}\widehat{K}\left(  k/R\right)  \left(
\widehat{\chi}_{\Omega}\left(  k\right)  \mp\widehat{H}_{R}\left(  k\right)
\right)  \exp\left(  2\pi ik\cdot x\right)  .
\]
Since $\left\vert \widehat{K}(\xi)\right\vert \leq1$, $\widehat{K}\left(
0\right)  =1$, $\widehat{K}(\xi)=0$ if $\left\vert \xi\right\vert \geq1$, and
since $\widehat{\chi}_{\Omega}\left(  0\right)  =\mu\left(  \Omega\right)  $,
then
\begin{gather*}
\mu\left(  \Omega\right)  -m^{-1}%
%TCIMACRO{\dsum _{j=1}^{m}}%
%BeginExpansion
{\displaystyle\sum_{j=1}^{m}}
%EndExpansion
\chi_{\Omega}(x_{j})\leq\mu\left(  \Omega\right)  -m^{-1}%
%TCIMACRO{\dsum _{j=1}^{m}}%
%BeginExpansion
{\displaystyle\sum_{j=1}^{m}}
%EndExpansion
A(x_{j})\\
=\mu\left(  \Omega\right)  -\sum_{k\in\mathbb{Z}^{d}}\widehat{K}\left(
k/R\right)  \left(  \widehat{\chi}_{\Omega}\left(  k\right)  -\widehat{H}%
_{R}\left(  k\right)  \right)  \left(  m^{-1}%
%TCIMACRO{\dsum _{j=1}^{m}}%
%BeginExpansion
{\displaystyle\sum_{j=1}^{m}}
%EndExpansion
\exp\left(  2\pi ik\cdot x_{j}\right)  \right) \\
\leq\left\vert \widehat{H}_{R}(0)\right\vert +\sum_{0<\left\vert k\right\vert
<R}\left(  \left\vert \widehat{\chi}_{\Omega}(k)\right\vert +\left\vert
\widehat{H}_{R}(k)\right\vert \right)  \left\vert m^{-1}%
%TCIMACRO{\dsum _{j=1}^{m}}%
%BeginExpansion
{\displaystyle\sum_{j=1}^{m}}
%EndExpansion
\exp\left(  2\pi ik\cdot x_{j}\right)  \right\vert .
\end{gather*}
Similarly
\begin{gather*}
-\mu\left(  \Omega\right)  +m^{-1}%
%TCIMACRO{\dsum _{j=1}^{m}}%
%BeginExpansion
{\displaystyle\sum_{j=1}^{m}}
%EndExpansion
\chi_{\Omega}(x_{j})\leq-\mu\left(  \Omega\right)  +m^{-1}%
%TCIMACRO{\dsum _{j=1}^{m}}%
%BeginExpansion
{\displaystyle\sum_{j=1}^{m}}
%EndExpansion
B(x_{j})\\
\leq\left\vert \widehat{H}_{R}(0)\right\vert +\sum_{0<\left\vert k\right\vert
<R}\left(  \left\vert \widehat{\chi}_{\Omega}(k)\right\vert +\left\vert
\widehat{H}_{R}(k)\right\vert \right)  \left\vert m^{-1}%
%TCIMACRO{\dsum _{j=1}^{m}}%
%BeginExpansion
{\displaystyle\sum_{j=1}^{m}}
%EndExpansion
\exp\left(  2\pi ik\cdot x_{j}\right)  \right\vert .
\end{gather*}

\end{proof}

In order to apply the above theorem one has to estimate the exponential sums
of point distributions and the Fourier transforms of domains. Motivated by the
above result and by the definition of Fourier dimension of a measurable set,
see e.g. Chapter 4.4 of \cite{Falconer}, it is possible to introduce the
classes of sets whose Fourier transform have a prescribed decay at infinity.
Given a measurable set $\Omega$ in the torus $\mathbb{T}^{d}$, assume that for
some $0\leq\alpha\leq(d+1)/2$ there exists a constant $c$ such that for every
$k\in\mathbb{Z}^{d}-\{0\}$,
\[
\left\vert
%TCIMACRO{\dint _{\mathbb{T}^{d}}}%
%BeginExpansion
{\displaystyle\int_{\mathbb{T}^{d}}}
%EndExpansion
\chi_{\Omega}\left(  x\right)  \exp\left(  -2\pi ik\cdot x\right)
dx\right\vert \leq c\left\vert k\right\vert ^{-\alpha}.
\]
Also, assume that for some $0\leq\beta\leq1$ and for every $R>0$,
\[
\left\vert
%TCIMACRO{\dint _{\mathbb{T}^{d}}}%
%BeginExpansion
{\displaystyle\int_{\mathbb{T}^{d}}}
%EndExpansion
\psi\left(  R\operatorname*{dist}\left(  x,\partial\Omega\right)  \right)
\exp\left(  -2\pi ik\cdot x\right)  dx\right\vert \leq\left\{
\begin{array}
[c]{l}%
c\left\vert k\right\vert ^{-\alpha}\text{ if }0<\left\vert k\right\vert <R,\\
cR^{-\beta}\;\text{if }k=0.
\end{array}
\right.
\]

Denote by $F\left(  \alpha,\beta,\Omega\right)  $ the smallest constant $c$
for which the above two inequalities hold. It turns out that the Fourier
transform of $\psi\left(  R\operatorname*{dist}\left(  x,\partial
\Omega\right)  \right)  $ is in some sense dominated by the one of
$\chi_{\Omega}(x)$. In particular, in a number of cases the second inequality
is a consequence of the first. Anyhow, in many cases it is possible to give
quite precise estimates for the constants $F\left(  \alpha,\beta
,\Omega\right)  $.

\begin{remark}
\label{5} If the first of the above inequalities holds for some $\alpha
>(d+1)/2$, or if the second inequality holds for some $\beta>1$, then either
$\Omega$ or $\mathbb{T}^{d}-\Omega$ has measure zero. It suffices to show this
when $(d+1)/2<\alpha<d/2+1$. For every $y$,
\begin{gather*}%
%TCIMACRO{\dint _{\mathbb{T}^{d}}}%
%BeginExpansion
{\displaystyle\int_{\mathbb{T}^{d}}}
%EndExpansion
\left\vert \chi_{\Omega}(x+y)-\chi_{\Omega}(x)\right\vert ^{2}dx=%
%TCIMACRO{\dsum _{k\in\mathbb{Z}^{d}}}%
%BeginExpansion
{\displaystyle\sum_{k\in\mathbb{Z}^{d}}}
%EndExpansion
\left\vert \left(  \exp(2\pi ik\cdot y)-1\right)  \widehat{\chi}_{\Omega
}(k)\right\vert ^{2}\\
\leq F\left(  \alpha,\beta,\Omega\right)  \left(  4\pi^{2}\left\vert
y\right\vert ^{2}%
%TCIMACRO{\dsum _{0<\left\vert k\right\vert <\left\vert y\right\vert ^{-1}}}%
%BeginExpansion
{\displaystyle\sum_{0<\left\vert k\right\vert <\left\vert y\right\vert ^{-1}}}
%EndExpansion
\left\vert k\right\vert ^{2-2\alpha}+4%
%TCIMACRO{\dsum _{\left\vert k\right\vert \geq\left\vert y\right\vert ^{-1}}}%
%BeginExpansion
{\displaystyle\sum_{\left\vert k\right\vert \geq\left\vert y\right\vert ^{-1}%
}}
%EndExpansion
\left\vert k\right\vert ^{-2\alpha}\right) \\
\leq cF\left(  \alpha,\beta,\Omega\right)  \left\vert y\right\vert
^{2\alpha-d}.
\end{gather*}
Hence for every $y$ and every $n\geq\left|  y\right|  $,
\begin{gather*}%
%TCIMACRO{\dint _{\mathbb{T}^{d}}}%
%BeginExpansion
{\displaystyle\int_{\mathbb{T}^{d}}}
%EndExpansion
\left|  \chi_{\Omega}(x+y)-\chi_{\Omega}(x)\right|  dx\\
\leq%
%TCIMACRO{\dsum _{j=1}^{n}}%
%BeginExpansion
{\displaystyle\sum_{j=1}^{n}}
%EndExpansion%
%TCIMACRO{\dint _{\mathbb{T}^{d}}}%
%BeginExpansion
{\displaystyle\int_{\mathbb{T}^{d}}}
%EndExpansion
\left|  \chi_{\Omega}\left(  x+jn^{-1}y\right)  -\chi_{\Omega}\left(
x+(j-1)n^{-1}y\right)  \right|  dx\\
=%
%TCIMACRO{\dsum _{j=1}^{n}}%
%BeginExpansion
{\displaystyle\sum_{j=1}^{n}}
%EndExpansion%
%TCIMACRO{\dint _{\mathbb{T}^{d}}}%
%BeginExpansion
{\displaystyle\int_{\mathbb{T}^{d}}}
%EndExpansion
\left|  \chi_{\Omega}\left(  x+jn^{-1}y\right)  -\chi_{\Omega}\left(
x+(j-1)n^{-1}y\right)  \right|  ^{2}dx\\
=n%
%TCIMACRO{\dint _{\mathbb{T}^{d}}}%
%BeginExpansion
{\displaystyle\int_{\mathbb{T}^{d}}}
%EndExpansion
\left|  \chi_{\Omega}\left(  x+n^{-1}y\right)  -\chi_{\Omega}\left(  x\right)
\right|  ^{2}dx\leq cF\left(  \alpha,\beta,\Omega\right)  \left|  y\right|
^{2\alpha-d}n^{d+1-2\alpha}.
\end{gather*}
This converges to zero when $n$ diverges. Hence, for every translation
$\Omega-y=\Omega$ up to sets with measure zero. Similarly, if $\beta>1$, then
either $\Omega$ or $\mathbb{T}^{d}-\Omega$ has measure zero. To see this, it
suffices to estimate the modulus of continuity of the function
\[
G(y)=%
%TCIMACRO{\dint _{\mathbb{T}^{d}}}%
%BeginExpansion
{\displaystyle\int_{\mathbb{T}^{d}}}
%EndExpansion
\left|  \chi_{\Omega}(x+y)-\chi_{\Omega}(x)\right|  dx.
\]
Indeed,
\begin{gather*}
\left\vert G(y)-G(z)\right\vert \\
\leq%
%TCIMACRO{\dint _{\mathbb{T}^{d}}}%
%BeginExpansion
{\displaystyle\int_{\mathbb{T}^{d}}}
%EndExpansion
\left\vert \chi_{\Omega}(x+y)-\chi_{\Omega}(x+z)\right\vert dx\leq\mu\left(
\left\{  \operatorname*{dist}\left(  x,\partial\Omega\right)  \leq\left\vert
y-z\right\vert \right\}  \right) \\
\leq\left(  \inf_{0\leq t\leq1}\left\{  \psi\left(  t\right)  \right\}
\right)  ^{-1}%
%TCIMACRO{\dint _{\mathbb{T}^{d}}}%
%BeginExpansion
{\displaystyle\int_{\mathbb{T}^{d}}}
%EndExpansion
\psi\left(  \left\vert y-z\right\vert ^{-1}\operatorname*{dist}\left(
x,\partial\Omega\right)  \right)  dx\leq cF\left(  \alpha,\beta,\Omega\right)
\left\vert y-z\right\vert ^{\beta}.
\end{gather*}
Finally, if $\beta>1$, then $G(y)$ is identically zero.
\end{remark}

\begin{remark}
The Fourier coefficients of $\psi\left(  R\operatorname*{dist}\left(
x,\partial\Omega\right)  \right)  $ on the torus $\mathbb{T}^{d}$ can be
evaluated by an integration over any set that tiles $\mathbb{R}^{d}$ via
$\mathbb{Z}^{d}$. In particular, one can integrate on the set of points $Q$ in
$\mathbb{R}^{d}$ for which the distance from $\partial\Omega+\mathbb{Z}^{d}$
is realized precisely by the distance from $\partial\Omega$. Since $\left\vert
\nabla\operatorname*{dist}\left(  x,\partial\Omega\right)  \right\vert =1$
when $\operatorname*{dist}\left(  x,\partial\Omega\right)  \neq0$, if
$\partial\Omega$ has measure zero, then the coarea formula gives
\begin{gather*}
\left\vert
%TCIMACRO{\dint _{Q\cap\left\{  \operatorname*{dist}\left(  x,\partial
%\Omega\right)  \leq\varepsilon\right\}  }}%
%BeginExpansion
{\displaystyle\int_{Q\cap\left\{  \operatorname*{dist}\left(  x,\partial
\Omega\right)  \leq\varepsilon\right\}  }}
%EndExpansion
\psi\left(  R\operatorname*{dist}\left(  x,\partial\Omega\right)  \right)
\exp\left(  -2\pi i\xi\cdot x\right)  dx\right\vert \\
\leq R^{-1}%
%TCIMACRO{\dint _{0}^{\varepsilon R}}%
%BeginExpansion
{\displaystyle\int_{0}^{\varepsilon R}}
%EndExpansion
\left\vert \psi\left(  t\right)  \right\vert \left\vert
%TCIMACRO{\dint _{\left\{  \operatorname*{dist}\left(  x,\partial\Omega\right)
%=R^{-1}t\right\}  }}%
%BeginExpansion
{\displaystyle\int_{\left\{  \operatorname*{dist}\left(  x,\partial
\Omega\right)  =R^{-1}t\right\}  }}
%EndExpansion
\exp\left(  -2\pi i\xi\cdot x\right)  dx\right\vert dt.
\end{gather*}
Moreover, since $Q$ has measure $1$,
\[
\left\vert
%TCIMACRO{\dint _{Q\cap\left\{  \operatorname*{dist}\left(  x,\partial
%\Omega\right)  >\varepsilon\right\}  }}%
%BeginExpansion
{\displaystyle\int_{Q\cap\left\{  \operatorname*{dist}\left(  x,\partial
\Omega\right)  >\varepsilon\right\}  }}
%EndExpansion
\psi\left(  R\operatorname*{dist}\left(  x,\partial\Omega\right)  \right)
\exp\left(  -2\pi i\xi\cdot x\right)  dx\right\vert \leq\sup_{t>\varepsilon
R}\left\vert \psi\left(  t\right)  \right\vert .
\]

In fact, one can eliminate the term $\sup_{t>\varepsilon R}\left\vert
\psi\left(  t\right)  \right\vert $ by integrating over all $Q$ rather than
$Q\cap\left\{  \operatorname*{dist}\left(  x,\partial\Omega\right)
\leq\varepsilon\right\}  $. However, in order to keep control of the level
sets $\left\{  \operatorname*{dist}\left(  x,\partial\Omega\right)
=t\right\}  $ it may be convenient to restrict to $t$ small.
\end{remark}

\begin{remark}
\label{7} It follows from classical estimates on oscillatory integrals with
non degenerate critical points that if a convex body has smooth boundary with
positive Gauss curvature, then
\begin{gather*}
\left\vert
%TCIMACRO{\dint _{\partial\Omega}}%
%BeginExpansion
{\displaystyle\int_{\partial\Omega}}
%EndExpansion
\exp\left(  -2\pi i\xi\cdot x\right)  dx\right\vert \leq c\left\vert
\xi\right\vert ^{-(d-1)/2},\\
\left\vert
%TCIMACRO{\dint _{\Omega}}%
%BeginExpansion
{\displaystyle\int_{\Omega}}
%EndExpansion
\exp\left(  -2\pi i\xi\cdot x\right)  dx\right\vert \leq c\left\vert
\xi\right\vert ^{-(d+1)/2}.
\end{gather*}

The constant $c$ can be bounded in terms of the smoothness and the minimum of
the curvature of the boundary. See e.g. \cite{Herz1,Herz2,Hlawka,Stein}. From
this estimates, with $\left\{  \operatorname*{dist}\left(  x,\partial
\Omega\right)  =t\right\}  $ instead of $\partial\Omega$, it easily follows
that $F\left(  (d+1)/2,1,\Omega\right)  $ is finite. If the curvature vanishes
of some order at some point, then the Fourier transform in directions normal
to these points has a worse decay at infinity. If some part of the boundary is
completely flat, then one can guarantee only a decay of order one.
\end{remark}

\begin{remark}
\label{8} If $0\leq\alpha\leq1$, define
\[
M\left(  \alpha,\Omega\right)  =\sup_{t>0}t^{-\alpha}\mu\left(  \left\{
\operatorname*{dist}\left(  x,\partial\Omega\right)  <t\right\}  \right)  .
\]
This is related to the upper Minkowski content of $\partial\Omega$, defined by%
\[
\lim\sup_{t\rightarrow0+}t^{-\alpha}\mu\left(  \left\{  \operatorname*{dist}%
\left(  x,\partial\Omega\right)  <t\right\}  \right)  .
\]
However, these two quantities can be quite different. In particular, if
$\partial\Omega$ contains a point $p$, then $M\left(  \alpha,\Omega\right)
\geq\mu\left(  \left\{  \operatorname*{dist}\left(  x,p\right)  <1\right\}
\right)  $. If $M\left(  \alpha,\Omega\right)  $ is finite, then
$\partial\Omega$ has Minkowski dimension at most $d-\alpha$. The domain is
regular if $\alpha=1$ and it is fractal if $0\leq\alpha<1$. The definition
makes sense also when $\alpha>1$, but in this case either $\Omega$ or
$\mathbb{T}^{d}-\Omega$ has measure zero. The proof is as in Remark \ref{5}.
It is well known that the decay of the Fourier transform of a domain can be
controlled in terms of these quantities. Indeed,
\begin{gather*}%
%TCIMACRO{\dint _{\mathbb{T}^{d}}}%
%BeginExpansion
{\displaystyle\int_{\mathbb{T}^{d}}}
%EndExpansion
\chi_{\Omega}(x)\exp(-2\pi ik\cdot x)dx\\
=-%
%TCIMACRO{\dint _{\mathbb{T}^{d}}}%
%BeginExpansion
{\displaystyle\int_{\mathbb{T}^{d}}}
%EndExpansion
\chi_{\Omega}(x)\exp(-2\pi ik\cdot\left(  x-2^{-1}\left\vert k\right\vert
^{-2}k\right)  )dx\\
=2^{-1}%
%TCIMACRO{\dint _{\mathbb{T}^{d}}}%
%BeginExpansion
{\displaystyle\int_{\mathbb{T}^{d}}}
%EndExpansion
\left(  \chi_{\Omega}(x)-\chi_{\Omega}(x+2^{-1}\left\vert k\right\vert
^{-2}k)\right)  \exp(-2\pi ik\cdot x)dx.
\end{gather*}
Then,
\begin{gather*}
\left|
%TCIMACRO{\dint _{\mathbb{T}^{d}}}%
%BeginExpansion
{\displaystyle\int_{\mathbb{T}^{d}}}
%EndExpansion
\chi_{\Omega}(x)\exp(-2\pi ik\cdot x)dx\right| \\
\leq2^{-1}%
%TCIMACRO{\dint _{\mathbb{T}^{d}}}%
%BeginExpansion
{\displaystyle\int_{\mathbb{T}^{d}}}
%EndExpansion
\left|  \chi_{\Omega}(x+2^{-1}\left|  k\right|  ^{-2}k)-\chi_{\Omega
}(x)\right|  dx\\
\leq2^{-1}\mu\left\{  \operatorname*{dist}\left(  x,\partial\Omega\right)
\leq2^{-1}\left|  k\right|  ^{-1}\right\}  \leq2^{-\alpha-1}M\left(
\alpha,\Omega\right)  \left|  k\right|  ^{-\alpha}.
\end{gather*}
Moreover, since $\psi\left(  t\right)  $ is positive and has fast decay at
infinity,
\begin{gather*}
\left\vert
%TCIMACRO{\dint _{\mathbb{T}^{d}}}%
%BeginExpansion
{\displaystyle\int_{\mathbb{T}^{d}}}
%EndExpansion
\psi\left(  R\operatorname*{dist}\left(  x,\partial\Omega\right)  \right)
\exp\left(  -2\pi ik\cdot x\right)  dx\right\vert \\
\leq%
%TCIMACRO{\dint _{\mathbb{T}^{d}}}%
%BeginExpansion
{\displaystyle\int_{\mathbb{T}^{d}}}
%EndExpansion
\psi\left(  R\operatorname*{dist}\left(  x,\partial\Omega\right)  \right)
dx\leq cM\left(  \alpha,\Omega\right)  R^{-\alpha}.
\end{gather*}
In particular, since $R^{-\alpha}\leq\left|  k\right|  ^{-\alpha}$ in the
range $0<\left|  k\right|  <R$, it follows that $F\left(  \alpha,\alpha
,\Omega\right)  \leq cM\left(  \alpha,\Omega\right)  $.
\end{remark}

In the following, the above theorem and remarks will be applied to the study
of discrepancies of lattices $m^{-1/d}\mathbb{Z}^{d}$ and arithmetic
progressions $\left\{  jx\right\}  _{j=1}^{m}$ in the torus $\mathbb{R}%
^{d}/\mathbb{Z}^{d}$, with multiples modulo $\mathbb{Z}^{d}$. In particular,
we will prove Theorems \ref{lattice}, \ref{ae} and \ref{polys} stated in the Introduction.

\begin{corollary}
\label{corlattice} {\ There exists a constant }$c${ with the following
property. Let }$m^{1/d}${ be a positive integer and let }$L(m)=m^{-1/d}%
\mathbb{Z}^{d}${ be the lattice of points }$m^{-1/d}\left(  g_{1}%
,...,g_{d}\right)  ${ with }$0\leq g_{j}\leq m^{1/d}-1${. Then }%
\[
\left\vert \mu(\Omega)-m^{-1}%
%TCIMACRO{\dsum _{x\in L(m)}}%
%BeginExpansion
{\displaystyle\sum_{x\in L(m)}}
%EndExpansion
\chi_{\Omega}\left(  x\right)  \right\vert \leq cF(\alpha,\beta,\Omega
)m^{-\beta/(d+\beta-\alpha)}.
\]

\end{corollary}

Theorem \ref{lattice} follows from the above Corollary by setting
$\beta=\alpha$, replacing $F(\alpha,\alpha,\Omega)$ with $M(\alpha,\Omega)$
and observing that the constant $c$ does not depend on the set $\Omega$.

\begin{proof}
The sum of a geometric progressions gives
\begin{gather*}
m^{-1}%
%TCIMACRO{\dsum _{x\in L(m)}}%
%BeginExpansion
{\displaystyle\sum_{x\in L(m)}}
%EndExpansion
\exp\left(  2\pi ik\cdot x\right) \\
=%
%TCIMACRO{\dprod _{n=1}^{d}}%
%BeginExpansion
{\displaystyle\prod_{n=1}^{d}}
%EndExpansion
\left(  m^{-1/d}%
%TCIMACRO{\dsum _{j=0}^{m^{1/d}-1}}%
%BeginExpansion
{\displaystyle\sum_{j=0}^{m^{1/d}-1}}
%EndExpansion
\exp\left(  2\pi im^{-1/d}jk_{n}\right)  \right)  =\left\{
\begin{array}
[c]{c}%
0\;\text{if }k\notin m^{1/d}\mathbb{Z}^{d}\text{,}\\
1\;\text{if }k\in m^{1/d}\mathbb{Z}^{d}\text{.}%
\end{array}
\right.
\end{gather*}
Hence, by Theorem \ref{4} and the definition of $F(\alpha,\beta,\Omega)$,
\begin{gather*}
\left\vert \mu(\Omega)-m^{-1}%
%TCIMACRO{\dsum _{x\in L(m)}}%
%BeginExpansion
{\displaystyle\sum_{x\in L(m)}}
%EndExpansion
\chi_{\Omega}\left(  x\right)  \right\vert \leq F(\alpha,\beta,\Omega)\left(
R^{-\beta}+2\sum_{0<\left\vert m^{1/d}k\right\vert <R}\left\vert
m^{1/d}k\right\vert ^{-\alpha}\right) \\
\leq F(\alpha,\beta,\Omega)\left(  R^{-\beta}+cm^{-1}R^{d-\alpha}\right)  .
\end{gather*}
Then the choice $R=m^{1/(d+\beta-\alpha)}$ gives the desired estimate.\ 
\end{proof}

\begin{corollary}
\label{10}{\ Given }$\varepsilon>0${, for almost every }$x${\ in }%
$\mathbb{T}^{d}${ there exists a constant }$c${ such that for every }$m>1${
and every }$\Omega${, }%
\[
\left\vert \mu(\Omega)-m^{-1}%
%TCIMACRO{\dsum _{j=1}^{m}}%
%BeginExpansion
{\displaystyle\sum_{j=1}^{m}}
%EndExpansion
\chi_{\Omega}(jx)\right\vert \leq cF(\alpha,\beta,\Omega)m^{-\beta/\left(
d+\beta-\alpha\right)  }\log^{\beta\left(  d+1+\varepsilon\right)  /\left(
d+\beta-\alpha\right)  }\left(  m\right)  .
\]

\end{corollary}

As before, Theorem \ref{ae} follows from the above Corollary by setting
$\beta=\alpha$, replacing $F(\alpha,\alpha,\Omega)$ with $M(\alpha,\Omega)$
and observing that the constant $c$ does not depend on the set $\Omega$.

\begin{proof}
Denoting by $\left\Vert t\right\Vert $ the distance of $t$ to the nearest
integer,
\[
\left\vert m^{-1}%
%TCIMACRO{\dsum \limits_{j=1}^{m}}%
%BeginExpansion
{\displaystyle\sum\limits_{j=1}^{m}}
%EndExpansion
\exp\left(  2\pi ijx\cdot k\right)  \right\vert =\left\vert \dfrac{\sin\left(
\pi mk\cdot x\right)  }{m\sin\left(  \pi k\cdot x\right)  }\right\vert
\leq\min\left\{  1,1/\left(  2m\left\Vert k\cdot x\right\Vert \right)
\right\}  .
\]
Hence, by Theorem \ref{4} and the definition of $F(\alpha,\beta,\Omega)$,
\begin{gather*}
\left\vert \mu(\Omega)-m^{-1}%
%TCIMACRO{\dsum _{j=1}^{m}}%
%BeginExpansion
{\displaystyle\sum_{j=1}^{m}}
%EndExpansion
\chi_{\Omega}(jx))\right\vert \\
\leq F(\alpha,\beta,\Omega)\left(  R^{-\beta}+2%
%TCIMACRO{\dsum _{0<\left\vert k\right\vert <R}}%
%BeginExpansion
{\displaystyle\sum_{0<\left\vert k\right\vert <R}}
%EndExpansion
\left\vert k\right\vert ^{-\alpha}\min\left\{  1,2^{-1}m^{-1}\left\Vert k\cdot
x\right\Vert ^{-1}\right\}  \right) \\
\leq F(\alpha,\beta,\Omega)\left(  R^{-\beta}+m^{-1}R^{d-\alpha}%
%TCIMACRO{\dsum _{0<\left\vert k\right\vert <R}}%
%BeginExpansion
{\displaystyle\sum_{0<\left\vert k\right\vert <R}}
%EndExpansion
\left\vert k\right\vert ^{-d}\left\Vert k\cdot x\right\Vert ^{-1}\right)  .
\end{gather*}
Finally, in \cite{Schmidt} it is proved that for almost every $x$ there exist
a $c$ which depends on $x$, such that for every $R>0$,
\[%
%TCIMACRO{\dsum _{0<\left|  k\right|  <R}}%
%BeginExpansion
{\displaystyle\sum_{0<\left|  k\right|  <R}}
%EndExpansion
\left|  k\right|  ^{-d}\left\|  k\cdot x\right\|  ^{-1}\leq c\log
^{d+1+\varepsilon}(1+R).
\]
The desired result follows choosing
\[
R=m^{1/\left(  d+\beta-\alpha\right)  }\log^{-\left(  d+1+\varepsilon\right)
/\left(  d+\beta-\alpha\right)  }\left(  m\right)  .
\]

\end{proof}

In particular, the discrepancy of a random sequence $\left\{  jx\right\}
_{j=1}^{m}$ with respect to any domain $\Omega$ is dominated by $cF\left(
(d+1)/2,1,\Omega\right)  m^{-2/(d+1)}\log^{2+\varepsilon}(m)$ and, by Remark
\ref{7}, when $\Omega$ is convex with smooth boundary and positive Gauss
curvature, then $F\left(  (d+1)/2,1,\Omega\right)  $ is finite. Similarly, by
Remark \ref{8}, the discrepancy is dominated by $cF\left(  1,1,\Omega\right)
m^{-1/d}\log^{\left(  d+1+\varepsilon\right)  /d}\left(  m\right)  $, and when
the boundary of the domain is $d-1$ dimensional, then $F\left(  1,1,\Omega
\right)  $ is finite. These results should be compared with well known upper
and lower estimates of the discrepancy with respect to convex regions due to
W.M. Schmidt and J.Beck. See e.g. \cite[Theorem 15 and Corollaries 17B, 18C
and 19F]{Beck-Chen}. In particular,
%for any given distribution of $m$ points there exists a convex set such that the discrepancy
%is bounded below by $cm^{-2/(d+1)}$, while
for any given compact convex set there exists an infinite sequence such that
the discrepancy is bounded above by $cm^{-(d+1)/(2d)}\log^{1/2} m$. The
definition of Fourier dimension does not capture polyhedra, except in the
trivial case $\alpha\leq1$. Indeed the decay of Fourier transforms of
polyhedra is not isotropic nor homogeneous. Anyhow, these Fourier transforms
can be computed explicitly and estimated quite precisely and, using these
estimates, one can give bounds for the discrepancy which are better than the
ones obtained above.

\begin{lemma}
{\ If }$\Omega${ is a polyhedron in }$\mathbb{R}^{d}${ with diameter }%
$\lambda${, then, }%
\[
\left|
%TCIMACRO{\dint _{\Omega}}%
%BeginExpansion
{\displaystyle\int_{\Omega}}
%EndExpansion
\exp\left(  -2\pi i\xi\cdot x\right)  dx\right|  \leq2\sum_{\Omega\left(
d\right)  \supset\ldots\supset\Omega\left(  1\right)  }%
%TCIMACRO{\dprod \limits_{j=1}^{d}}%
%BeginExpansion
{\displaystyle\prod\limits_{j=1}^{d}}
%EndExpansion
\min\left\{  \lambda,\left(  2\pi\left|  P_{\Omega\left(  j\right)  }%
\xi\right|  \right)  ^{-1}\right\}  .
\]
{The sum is taken over all possible decreasing chains of }$j${ dimensional
faces} $\Omega(j)$ { of }$\Omega${ and }$P_{\Omega\left(  j\right)  }${\ is
the orthogonal projection on the }$j${\ dimensional subspace parallel to
}$\Omega\left(  j\right)  ${.}
\end{lemma}

\begin{proof}
The Fourier transform of a polyhedron can be computed explicitly, but here we
are only interested in precise estimates of its size with control on all
constants involved. If $\Omega\left(  j\right)  $ is a $j$ dimensional face of
$\Omega$ with $P_{\Omega\left(  j\right)  }\xi\neq0$ and if $n\left(
x\right)  $ is the outgoing normal to the boundary $\partial\Omega\left(
j\right)  $ at the point $x$, the divergence theorem gives
\[%
%TCIMACRO{\dint _{\Omega(j)}}%
%BeginExpansion
{\displaystyle\int_{\Omega(j)}}
%EndExpansion
\exp\left(  -2\pi i\xi\cdot x\right)  dx=%
%TCIMACRO{\dint _{\partial\Omega\left(  j\right)  }}%
%BeginExpansion
{\displaystyle\int_{\partial\Omega\left(  j\right)  }}
%EndExpansion
\frac{in\left(  x\right)  \cdot P_{\Omega\left(  j\right)  }\xi}{2\pi\left|
P_{\Omega\left(  j\right)  }\xi\right|  ^{2}}\exp\left(  -2\pi i\xi\cdot
x\right)  dx.
\]
Moreover, if $\lambda$ is the diameter and $\mu\left(  \Omega(j)\right)  $ is
the $j$ dimensional measure of $\Omega(j)$, one always has the trivial
estimate
\[
\left|
%TCIMACRO{\dint _{\Omega(j)}}%
%BeginExpansion
{\displaystyle\int_{\Omega(j)}}
%EndExpansion
\exp\left(  -2\pi i\xi\cdot x\right)  dx\right|  \leq\mu\left(  \Omega
(j)\right)  \leq\lambda^{j}.
\]
Hence, if $\Omega\left(  j-1\right)  $ are the $j-1$ dimensional faces of
$\Omega\left(  j\right)  $,
\begin{gather*}
\left\vert
%TCIMACRO{\dint _{\Omega(j)}}%
%BeginExpansion
{\displaystyle\int_{\Omega(j)}}
%EndExpansion
\exp\left(  -2\pi i\xi\cdot x\right)  dx\right\vert \leq\\
\left\{
\begin{tabular}
[c]{ll}%
$\lambda^{j}$ & $\text{if }2\pi\left\vert P_{\Omega\left(  j\right)  }%
\xi\right\vert <1/\lambda\text{,}$\\
$\dfrac{1}{2\pi\left\vert P_{\Omega\left(  j\right)  }\xi\right\vert }%
\sum\limits_{\Omega\left(  j-1\right)  \subset\Omega\left(  j\right)
}\left\vert
%TCIMACRO{\dint _{\Omega\left(  j-1\right)  }}%
%BeginExpansion
{\displaystyle\int_{\Omega\left(  j-1\right)  }}
%EndExpansion
\exp\left(  -2\pi i\xi\cdot x\right)  dx\right\vert $ & $\text{if }%
2\pi\left\vert P_{\Omega\left(  j\right)  }\xi\right\vert \geq1/\lambda
\text{.}$%
\end{tabular}
\right.
\end{gather*}
Iterating, one can decompose the integral over $\Omega=\Omega\left(  d\right)
$ into a sum of integrals over chains of faces $\Omega\left(  d\right)
\supset\Omega\left(  d-1\right)  \supset...$ and this gives
\begin{gather*}
\left\vert
%TCIMACRO{\dint _{\Omega}}%
%BeginExpansion
{\displaystyle\int_{\Omega}}
%EndExpansion
\exp\left(  -2\pi i\xi\cdot x\right)  dx\right\vert \\
\leq\sum_{\Omega\left(  d\right)  \supset\Omega\left(  d-1\right)
\supset\ldots\supset\Omega\left(  s\right)  }\lambda^{s}\prod\limits_{j=s+1}%
^{d}\left(  2\pi\left\vert P_{\Omega\left(  j\right)  }\xi\right\vert \right)
^{-1}\\
+2\sum_{\Omega\left(  d\right)  \supset\Omega\left(  d-1\right)  \supset
\ldots\supset\Omega\left(  1\right)  }\prod\limits_{j=1}^{d}\left(
2\pi\left\vert P_{\Omega\left(  j\right)  }\xi\right\vert \right)  ^{-1}.
\end{gather*}
The first sum is over all chains of faces with $\left(  2\pi\left\vert
P_{\Omega\left(  j\right)  }\xi\right\vert \right)  ^{-1}\leq\lambda$ for
$1\leq s<j\leq d$ and $\left(  2\pi\left\vert P_{\Omega\left(  s\right)  }%
\xi\right\vert \right)  ^{-1}>\lambda$, while the second sum is over all
chains with $\left(  2\pi\left\vert P_{\Omega\left(  j\right)  }\xi\right\vert
\right)  ^{-1}\leq\lambda$ for all $1\leq j\leq d$. Finally, in the first
sum,
\[
\lambda^{s}\prod\limits_{j=s+1}^{d}\left(  2\pi\left\vert P_{\Omega\left(
j\right)  }\xi\right\vert \right)  ^{-1}=%
%TCIMACRO{\dprod \limits_{j=1}^{d}}%
%BeginExpansion
{\displaystyle\prod\limits_{j=1}^{d}}
%EndExpansion
\min\left\{  \lambda,\left(  2\pi\left\vert P_{\Omega\left(  j\right)  }%
\xi\right\vert \right)  ^{-1}\right\}  .
\]
Indeed, since $\left|  P_{\Omega\left(  j\right)  }\xi\right|  $ is increasing
in $j$, the terms $\min\left\{  \lambda,\left(  2\pi\left|  P_{\Omega\left(
j\right)  }\xi\right|  \right)  ^{-1}\right\}  $ are equal to $\lambda$ when
$1\leq j\leq s$ and equal to $\left(  2\pi\left|  P_{\Omega\left(  j\right)
}\xi\right|  \right)  ^{-1}$ when $s<j\leq d$. Similarly, in the second sum,
\[
\prod\limits_{j=1}^{d}\left(  2\pi\left|  P_{\Omega\left(  j\right)  }%
\xi\right|  \right)  ^{-1}=%
%TCIMACRO{\dprod \limits_{j=1}^{d}}%
%BeginExpansion
{\displaystyle\prod\limits_{j=1}^{d}}
%EndExpansion
\min\left\{  \lambda,\left(  2\pi\left|  P_{\Omega\left(  j\right)  }%
\xi\right|  \right)  ^{-1}\right\}  .
\]
Observe that when $\xi=0$ the formula gives $\lambda^{d}$, while the exact
value of the integral is the volume of $\Omega$.
\end{proof}

\begin{lemma}
{\ Let }$\Omega${ be a convex polyhedron in }$\mathbb{R}^{d}${ with diameter
}$\lambda${. For any }$j=1,2,...,d-1${, let }$\left\{  A\left(  j\right)
\right\}  ${\ be the collection of all }$j${\ dimensional subspaces which are
intersections of a number of subspaces parallel to the faces of }$\Omega${.
Finally, let }$\psi(t)${\ be a function on }$0\leq t<+\infty${\ with fast
decay at infinity. Then there exists a positive constant }$c${, which depends
on }$d${ and }$\psi(t)${, but not on }$\Omega${, such that for every }$R>0${,
}%
\begin{gather*}
\left\vert
%TCIMACRO{\dint _{\mathbb{R}^{d}}}%
%BeginExpansion
{\displaystyle\int_{\mathbb{R}^{d}}}
%EndExpansion
\psi\left(  R\operatorname*{dist}\left(  x,\partial\Omega\right)  \right)
\exp\left(  -2\pi i\xi\cdot x\right)  dx\right\vert \\
\leq c%
%TCIMACRO{\dsum _{j=0}^{d-1}}%
%BeginExpansion
{\displaystyle\sum_{j=0}^{d-1}}
%EndExpansion%
%TCIMACRO{\dsum _{A\left(  j\right)  \supset...\supset A\left(  1\right)  }}%
%BeginExpansion
{\displaystyle\sum_{A\left(  j\right)  \supset...\supset A\left(  1\right)  }}
%EndExpansion
R^{j-d}%
%TCIMACRO{\dprod _{k=1}^{j}}%
%BeginExpansion
{\displaystyle\prod_{k=1}^{j}}
%EndExpansion
\min\left\{  \lambda,\left(  2\pi\left\vert P_{A\left(  k\right)  }%
\xi\right\vert \right)  ^{-1}\right\}  .
\end{gather*}
{When }$j=0${ the inner sum of products is intended to be the number of
vertices of the polyhedron, when }$1\leq j\leq d-1${ the inner sum is taken
over all possible decreasing chains of }$j${ dimensional subspaces }$\left\{
A\left(  j\right)  \right\}  ${ and }$P_{A\left(  j\right)  }${\ is the
orthogonal projection on }$A\left(  j\right)  ${.}
\end{lemma}

\begin{proof}
Since $\left|  \nabla\operatorname*{dist}\left(  x,\partial\Omega\right)
\right|  =1$, the coarea formula gives
\begin{gather*}%
%TCIMACRO{\dint _{\mathbb{R}^{d}}}%
%BeginExpansion
{\displaystyle\int_{\mathbb{R}^{d}}}
%EndExpansion
\psi\left(  R\operatorname*{dist}\left(  x,\partial\Omega\right)  \right)
\exp\left(  -2\pi i\xi\cdot x\right)  dx\\
=%
%TCIMACRO{\dint _{0}^{+\infty}}%
%BeginExpansion
{\displaystyle\int_{0}^{+\infty}}
%EndExpansion
\left(
%TCIMACRO{\dint _{\left\{  \operatorname*{dist}\left(  x,\partial\Omega\right)
%=t\right\}  }}%
%BeginExpansion
{\displaystyle\int_{\left\{  \operatorname*{dist}\left(  x,\partial
\Omega\right)  =t\right\}  }}
%EndExpansion
\exp\left(  -2\pi i\xi\cdot x\right)  dx\right)  \psi\left(  Rt\right)  dt.
\end{gather*}

We consider separately the level sets inside and outside $\Omega$. The level
sets $\left\{  \operatorname*{dist}\left(  x,\partial\Omega\right)
=t\right\}  \cap\Omega$ are polyhedra with diameter at most $\lambda$ and
faces parallel to some $\left\{  A\left(  j\right)  \right\}  $. We emphasize
that some of the $\left\{  A\left(  j\right)  \right\}  $ may not be parallel
to any $\left\{  \Omega\left(  j\right)  \right\}  $. Anyhow, as in the
previous lemma,
\begin{gather*}
\left|
%TCIMACRO{\dint _{0}^{+\infty}}%
%BeginExpansion
{\displaystyle\int_{0}^{+\infty}}
%EndExpansion
\left(
%TCIMACRO{\dint _{\left\{  \operatorname*{dist}\left(  x,\partial\Omega\right)
%=t\right\}  \cap\Omega}}%
%BeginExpansion
{\displaystyle\int_{\left\{  \operatorname*{dist}\left(  x,\partial
\Omega\right)  =t\right\}  \cap\Omega}}
%EndExpansion
\exp\left(  -2\pi i\xi\cdot x\right)  dx\right)  \psi\left(  Rt\right)
dt\right| \\
\leq%
%TCIMACRO{\dsum _{A\left(  d-1\right)  \supset\ldots\supset A\left(  1\right)
%}}%
%BeginExpansion
{\displaystyle\sum_{A\left(  d-1\right)  \supset\ldots\supset A\left(
1\right)  }}
%EndExpansion
\left(  2^{d}%
%TCIMACRO{\dint _{0}^{+\infty}}%
%BeginExpansion
{\displaystyle\int_{0}^{+\infty}}
%EndExpansion
\left|  \psi\left(  Rt\right)  \right|  dt\right)
%TCIMACRO{\dprod _{j=1}^{d-1}}%
%BeginExpansion
{\displaystyle\prod_{j=1}^{d-1}}
%EndExpansion
\min\left\{  \lambda,\left(  2\pi\left|  P_{A\left(  j\right)  }\xi\right|
\right)  ^{-1}\right\}  .
\end{gather*}

As in Steiner formula, the outer level sets $\left\{  \operatorname*{dist}%
\left(  x,\partial\Omega\right)  =t\right\}  -\Omega$ are union of sums of $j$
dimensional faces $\left\{  \Omega\left(  j\right)  \right\}  $ and portions
of $d-j-1$ dimensional spherical surfaces of radius $t$. Hence, if
$\omega(d-j-1)$ denotes the measure of the $d-j-1$ dimensional spherical
surface of unit radius,
\begin{gather*}
\left\vert
%TCIMACRO{\dint _{0}^{+\infty}}%
%BeginExpansion
{\displaystyle\int_{0}^{+\infty}}
%EndExpansion
\left(
%TCIMACRO{\dint _{\left\{  \operatorname*{dist}\left(  x,\partial\Omega\right)
%=t\right\}  -\Omega}}%
%BeginExpansion
{\displaystyle\int_{\left\{  \operatorname*{dist}\left(  x,\partial
\Omega\right)  =t\right\}  -\Omega}}
%EndExpansion
\exp\left(  -2\pi i\xi\cdot x\right)  dx\right)  \psi\left(  Rt\right)
dt\right\vert \leq\\%
%TCIMACRO{\dsum _{j=0}^{d-1}}%
%BeginExpansion
{\displaystyle\sum_{j=0}^{d-1}}
%EndExpansion
\left(  \omega(d-j-1)%
%TCIMACRO{\dint _{0}^{+\infty}}%
%BeginExpansion
{\displaystyle\int_{0}^{+\infty}}
%EndExpansion
\left\vert \psi\left(  Rt\right)  \right\vert t^{d-j-1}dt\right)  \times\\
2%
%TCIMACRO{\dsum _{\Omega\left(  j\right)  \supset...\supset\Omega\left(
%1\right)  }}%
%BeginExpansion
{\displaystyle\sum_{\Omega\left(  j\right)  \supset...\supset\Omega\left(
1\right)  }}
%EndExpansion%
%TCIMACRO{\dprod _{k=1}^{j}}%
%BeginExpansion
{\displaystyle\prod_{k=1}^{j}}
%EndExpansion
\min\left\{  \lambda,\left(  2\pi\left\vert P_{\Omega\left(  k\right)  }%
\xi\right\vert \right)  ^{-1}\right\}  .
\end{gather*}
Observe that when $\xi=0$ the formula gives $c\left(  R^{-d}+\lambda
^{d-1}R^{-1}\right)  $, with a constant $c$ independent of the polyhedron.
\end{proof}

\begin{theorem}
\label{13}{\ Given a finite collection of }$d-1${ dimensional hyperspaces }%
$X${ in }$\mathbb{R}^{d}${, let }$P(X)${ be the collection of all convex
polyhedra with diameter smaller than }$1-\varepsilon${ and facets parallel to
elements of }$X${. If }$\left\{  A\left(  d\right)  \supset...\supset A\left(
1\right)  \right\}  ${\ is the collection of all possible decreasing chains of
}$j${ dimensional subspaces obtained by intersection of }$\mathbb{R}^{d}${ and
a number of hyperplanes in }$X${, define }%
\[
\Phi\left(  \xi\right)  =%
%TCIMACRO{\dsum _{A\left(  d\right)  \supset...\supset A\left(  1\right)  }}%
%BeginExpansion
{\displaystyle\sum_{A\left(  d\right)  \supset...\supset A\left(  1\right)  }}
%EndExpansion
\,%
%TCIMACRO{\dprod _{j=1}^{d}}%
%BeginExpansion
{\displaystyle\prod_{j=1}^{d}}
%EndExpansion
\min\left\{  1,\left(  2\pi\left\vert P_{A\left(  j\right)  }\xi\right\vert
\right)  ^{-1}\right\}  .
\]
{Finally, if }$\left\{  x_{j}\right\}  _{j=1}^{m}${ is a sequence of points in
the torus, define }%
\[
\Psi\left(  \xi\right)  =\left|  m^{-1}%
%TCIMACRO{\dsum _{j=1}^{m}}%
%BeginExpansion
{\displaystyle\sum_{j=1}^{m}}
%EndExpansion
\exp\left(  2\pi i\xi\cdot x_{j}\right)  \right|  .
\]
{Then, there exists a positive constant }$c${, which depends only on the space
dimension }$d${ and the upper bound }$1-\varepsilon${ of diameters of
polyhedra, such that for every }$R>0${, }%
\[
\sup_{\Omega\in P(X)}\left|  \mu(\Omega)-m^{-1}%
%TCIMACRO{\dsum _{j=1}^{m}}%
%BeginExpansion
{\displaystyle\sum_{j=1}^{m}}
%EndExpansion
\chi_{\Omega}(x_{j})\right|  \leq c\left(  R^{-1}+%
%TCIMACRO{\dsum _{0<\left|  k\right|  <R}}%
%BeginExpansion
{\displaystyle\sum_{0<\left|  k\right|  <R}}
%EndExpansion
\Phi\left(  k\right)  \Psi\left(  k\right)  \right)  .
\]

\end{theorem}

\begin{proof}
This follows from Theorem \ref{4} and the previous lemmas. It suffices to
replace $\lambda\leq1-\varepsilon$ with $1$ and $R^{j-d}$ with $\left\vert
k\right\vert ^{j-d}$ in the range $0<\left\vert k\right\vert <R$.
\end{proof}

For example, if $X$ is the collection of hyperplanes $\left\{  x_{j}%
=0\right\}  $, then $P(X)$ is the collection of all boxes $\left\{  a_{j}\leq
x_{j}\leq b_{j}\right\}  $ with
%$b_{j}-a_{j}<1-\varepsilon$
diameter smaller than $1-\varepsilon$ and the above is an estimate of
discrepancy with respect to boxes.

\begin{corollary}
\label{corgoodlatticepoint}{\ Given a finite collection of hyperspaces }$X${
and a prime number }$m${, there exists a lattice point }$g=\left(
g_{1},...g_{d}\right)  ${ in }$\mathbb{Z}^{d}${ with }$1\leq g_{j}\leq m-1${,
such that }%
\[
\sup_{\Omega\in P(X)}\left\vert \mu(\Omega)-m^{-1}%
%TCIMACRO{\dsum _{j=1}^{m}}%
%BeginExpansion
{\displaystyle\sum_{j=1}^{m}}
%EndExpansion
\chi_{\Omega}\left(  jm^{-1}g\right)  \right\vert \leq cm^{-1}\log^{d}(m).
\]
{The constant }$c${ depends only on the dimension and on the cardinality of
the set of chains of subspaces }$A\left(  d\right)  \supset\ldots\supset
A\left(  1\right)  ${ generated by }$X${.}
\end{corollary}

Theorem \ref{polys} follows from the above Corollary, along with the
observation that a polyhedron in the torus with facets parallel to elements of
$X$ and diameter smaller than $1$ can be seen as the union of a finite number
(which depends only on $X$) of polyhedra with facets parallel to elements of
$X$ and diameter smaller than, say, $1/2$.

\begin{proof}
A preliminary result is needed.

\noindent\textit{Claim.} {There is a constant }$c${, depending only on the
dimension }$d${, such that, for any decreasing chain of subspaces }$A\left(
d\right)  \supset\ldots\supset A\left(  1\right)  ${ and for any }$R>0${, }%
\[%
%TCIMACRO{\dsum _{1\leq\left\vert k\right\vert \leq R}}%
%BeginExpansion
{\displaystyle\sum_{1\leq\left\vert k\right\vert \leq R}}
%EndExpansion
\,%
%TCIMACRO{\dprod _{j=1}^{d}}%
%BeginExpansion
{\displaystyle\prod_{j=1}^{d}}
%EndExpansion
\min\left\{  1,\left(  2\pi\left\vert P_{A\left(  j\right)  }k\right\vert
\right)  ^{-1}\right\}  \leq c\log^{d}\left(  2+R\right)  .
\]

Indeed, let $\left\{  a_{1},\ldots,a_{d}\right\}  $ be an orthonormal basis of
$\mathbb{R}^{d}$ such that $\left\{  a_{1},\ldots,a_{j}\right\}  $ generates
$A\left(  j\right)  $ for $j=1,...,d$. Since $\left\vert P_{A\left(  j\right)
}k\right\vert \geq\left\vert k\cdot a_{j}\right\vert $, it suffices to
estimate
\[%
%TCIMACRO{\dsum _{1\leq\left\vert k\right\vert \leq R}}%
%BeginExpansion
{\displaystyle\sum_{1\leq\left\vert k\right\vert \leq R}}
%EndExpansion
\,%
%TCIMACRO{\dprod _{j=1}^{d}}%
%BeginExpansion
{\displaystyle\prod_{j=1}^{d}}
%EndExpansion
\min\left\{  1,\left\vert k\cdot a_{j}\right\vert ^{-1}\right\}  .
\]

Of course the idea is to replace the sum over a discrete variable with an
integral over a continuous variable. Let $Q(k)$ be the cube centered at the
integer point $k$ with sides of length $1$ parallel to the orthogonal axes
$\left\{  a_{j}\right\}  $ and let $P(k)$ be the cube centered at $k$ with
sides $3$. Since $\left\vert x\cdot a_{j}\right\vert $ is the distance of $x$
from the hyperplane of equation $x\cdot a_{j}=0$, if this hyperplane crosses
$P(k)$, then $\left\vert x\cdot a_{j}\right\vert \leq2\sqrt{d}$ for any $x$ in
$Q(k)$ and therefore
\[
\min\left\{  1,\left\vert k\cdot a_{j}\right\vert ^{-1}\right\}  \leq
1=\min\left\{  1,2\sqrt{d}\left\vert x\cdot a_{j}\right\vert ^{-1}\right\}  .
\]
If this hyperplane does not cross $P(k)$, then for any $x$ in $Q(k)$,
\begin{gather*}
\min\left\{  1,\left|  k\cdot a_{j}\right|  ^{-1}\right\}  =\min\left\{
1,\frac{\left|  x\cdot a_{j}\right|  }{\left|  k\cdot a_{j}\right|  \left|
x\cdot a_{j}\right|  }\right\}  .\\
\leq\min\left\{  1,\frac{\left|  k\cdot a_{j}\right|  +\sqrt{d}/2}{\left|
k\cdot a_{j}\right|  \left|  x\cdot a_{j}\right|  }\right\}  \leq\min\left\{
1,\left(  1+\sqrt{d}/2\right)  \left|  x\cdot a_{j}\right|  ^{-1}\right\}  .
\end{gather*}
In the overall, for any $k$ and any $x\in Q(k)$,%
\[%
%TCIMACRO{\dprod \limits_{j=1}^{d}}%
%BeginExpansion
{\displaystyle\prod\limits_{j=1}^{d}}
%EndExpansion
\min\left(  1,\left|  k\cdot a_{j}\right|  ^{-1}\right)  \leq%
%TCIMACRO{\dprod \limits_{j=1}^{d}}%
%BeginExpansion
{\displaystyle\prod\limits_{j=1}^{d}}
%EndExpansion
\min\left(  1,2\sqrt{d}\left|  x\cdot a_{j}\right|  ^{-1}\right)  .
\]
Hence,
\begin{gather*}%
%TCIMACRO{\dsum _{1\leq\left\vert k\right\vert \leq R}}%
%BeginExpansion
{\displaystyle\sum_{1\leq\left\vert k\right\vert \leq R}}
%EndExpansion%
%TCIMACRO{\dprod _{j=1}^{d}}%
%BeginExpansion
{\displaystyle\prod_{j=1}^{d}}
%EndExpansion
\min\left\{  1,\left\vert k\cdot a_{j}\right\vert ^{-1}\right\} \\
\leq%
%TCIMACRO{\dsum _{1\leq\left\vert k\right\vert \leq R}}%
%BeginExpansion
{\displaystyle\sum_{1\leq\left\vert k\right\vert \leq R}}
%EndExpansion
\int_{Q(k)}%
%TCIMACRO{\dprod _{j=1}^{d}}%
%BeginExpansion
{\displaystyle\prod_{j=1}^{d}}
%EndExpansion
\min\left\{  1,2\sqrt{d}\left\vert x\cdot a_{j}\right\vert ^{-1}\right\}  dx\\
\leq\int_{\left\{  \left\vert x\cdot a_{j}\right\vert \leq R+1/2\right\}  }%
%TCIMACRO{\dprod _{j=1}^{d}}%
%BeginExpansion
{\displaystyle\prod_{j=1}^{d}}
%EndExpansion
\min\left\{  1,2\sqrt{d}\left\vert x\cdot a_{j}\right\vert ^{-1}\right\}  dx\\
=\left(  2\int_{0}^{R+1/2}\min\left\{  1,2\sqrt{d}t^{-1}\right\}  dt\right)
^{d}\leq c\log^{d}\left(  2+R\right)  .
\end{gather*}

Now comes the proof of the Corollary. The sum of a geometric progression
gives
\[
m^{-1}%
%TCIMACRO{\dsum _{j=1}^{m}}%
%BeginExpansion
{\displaystyle\sum_{j=1}^{m}}
%EndExpansion
\exp\left(  2\pi ijm^{-1}g\cdot k\right)  =\dfrac{\sin\left(  \pi g\cdot
k\right)  }{m\sin\left(  \pi m^{-1}g\cdot k\right)  }\exp\left(  \pi
i(m+1)m^{-1}g\cdot k\right)  .
\]
This exponential sum is $0$ or $1$ according to $g\cdot k\neq
0\;(\operatorname{mod}m)$ or $g\cdot k\equiv0\;(\operatorname{mod}m)$. Hence,
by Theorem \ref{13} with $R=m$,
\[
\sup_{\Omega\in P(X)}\left|  \mu(\Omega)-m^{-1}%
%TCIMACRO{\dsum _{j=1}^{m}}%
%BeginExpansion
{\displaystyle\sum_{j=1}^{m}}
%EndExpansion
\chi_{\Omega}\left(  jm^{-1}g\right)  \right|  \leq c\left(  m^{-1}+%
%TCIMACRO{\dsum _{0<\left|  k\right|  <m,\;g\cdot k\equiv0\;(\operatorname{mod}%
%m)}}%
%BeginExpansion
{\displaystyle\sum_{0<\left|  k\right|  <m,\;g\cdot k\equiv
0\;(\operatorname{mod}m)}}
%EndExpansion
\Phi\left(  k\right)  \right)  .
\]

The heuristic behind the existence of good lattice points with the desired
properties is that the ratio of $k$'s which satisfy the congruence $g\cdot
k\equiv0\;(\operatorname{mod}m)$ is $m^{-1}$ and the sum over the $k$'s with
$g\cdot k\equiv0\;(\operatorname{mod}m)$ is $m^{-1}$ times the sum over all
$k$. This heuristic principle can be made rigorous by an averaging procedure,
as in Theorem 5.7 of \cite{Kuipers-Niederreiter}. In order to satisfy the
congruence $g_{1}k_{1}+...+g_{d}k_{d}\equiv0\;(\operatorname{mod}m)$, if
$k_{i}\neq0\;(\operatorname{mod}m)$, then one can take $g_{j}$ arbitrary for
$j\neq i$, the remaining $g_{i}$ being uniquely determined in the residue
class. Hence, by the above claim,
\begin{gather*}
(m-1)^{-d}%
%TCIMACRO{\dsum _{1\leq g_{j}\leq m-1}}%
%BeginExpansion
{\displaystyle\sum_{1\leq g_{j}\leq m-1}}
%EndExpansion
\left(
%TCIMACRO{\dsum _{0<\left\vert k\right\vert <m,\;g\cdot k\equiv
%0\;(\operatorname{mod}m)}}%
%BeginExpansion
{\displaystyle\sum_{0<\left\vert k\right\vert <m,\;g\cdot k\equiv
0\;(\operatorname{mod}m)}}
%EndExpansion
\Phi\left(  k\right)  \right) \\
=%
%TCIMACRO{\dsum _{0<\left\vert k\right\vert <m}}%
%BeginExpansion
{\displaystyle\sum_{0<\left\vert k\right\vert <m}}
%EndExpansion
\Phi\left(  k\right)  \left(  (m-1)^{-d}%
%TCIMACRO{\dsum _{1\leq g_{j}\leq m-1,\;g\cdot k\equiv0\;(\operatorname{mod}%
%m)}}%
%BeginExpansion
{\displaystyle\sum_{1\leq g_{j}\leq m-1,\;g\cdot k\equiv0\;(\operatorname{mod}%
m)}}
%EndExpansion
1\right) \\
\leq(m-1)^{-1}%
%TCIMACRO{\dsum _{0<\left\vert k\right\vert <m}}%
%BeginExpansion
{\displaystyle\sum_{0<\left\vert k\right\vert <m}}
%EndExpansion
\Phi\left(  k\right)  \leq cm^{-1}\log^{d}(m).
\end{gather*}

Observe that the constant $c$ is the product of the constant in the above
claim and the cardinality of the set of chains of subspaces generated by $X$.
In particular, there exists $g$ such that
\[%
%TCIMACRO{\dsum _{0<\left\vert k\right\vert <m,\;g\cdot k\equiv
%0\;(\operatorname{mod}m)}}%
%BeginExpansion
{\displaystyle\sum_{0<\left\vert k\right\vert <m,\;g\cdot k\equiv
0\;(\operatorname{mod}m)}}
%EndExpansion
\Phi\left(  k\right)  \leq cm^{-1}\log^{d}(m).
\]

\end{proof}

\begin{corollary}
{Given a finite collection of hyperspaces }$X${ and }$\varepsilon>0${, then
for every }$\varepsilon>0${\ and almost every }$x${\ in }$\mathbb{T}^{d}%
${\ there exists a positive constant }$c${\ such that for every }$m>1$,
\[
\sup_{\Omega\in P(X)}\left|  \mu(\Omega)-m^{-1}%
%TCIMACRO{\dsum _{j=1}^{m}}%
%BeginExpansion
{\displaystyle\sum_{j=1}^{m}}
%EndExpansion
\chi_{\Omega}\left(  jx\right)  \right|  \leq cm^{-1}\log^{d+1+\varepsilon
}(m).
\]

\end{corollary}

\begin{proof}
As in Corollary \ref{10}, this follows from Theorem \ref{13} and an adaptation
of \cite{Schmidt}.
\end{proof}

The above corollaries improve and extend a two dimensional result in
\cite[Theorem 4D]{Beck}, where it is proved that the discrepancy of $m$ points
with respect to polygons is dominated by $m^{-1}\log^{4+\varepsilon}\left(
m\right)  $. They should also be compared with the discrepancy of lattice
points in \cite{Brandolini-Colzani-Travaglini}.

\section{Approximation and discrepancy on manifolds}

Let $\mathcal{M}$ be a smooth $d$ dimensional compact manifold without
boundary, with Riemannian distance $\operatorname*{dist}(x,y)$ and measure
$\mu$ normalized so that $\mu\left(  \mathcal{M}\right)  =1$. The Laplace
Beltrami operator $\Delta$ on $\mathcal{M}$ has eigenvalues $\left\{
\lambda^{2}\right\}  $, counted with appropriate multiplicity, and a complete
orthonormal system of eigenfunctions $\left\{  \varphi_{\lambda}(x)\right\}
$. To every function in $L^{2}(\mathcal{M},d\mu)$ one can associate a Fourier
transform and a Fourier series,
\[
\widehat{f}(\lambda)=%
%TCIMACRO{\dint _{\mathcal{M}}}%
%BeginExpansion
{\displaystyle\int_{\mathcal{M}}}
%EndExpansion
f(y)\overline{\varphi_{\lambda}(y)}d\mu(y),\quad f(x)=%
%TCIMACRO{\dsum \limits_{\lambda}}%
%BeginExpansion
{\displaystyle\sum\limits_{\lambda}}
%EndExpansion
\widehat{f}(\lambda)\varphi_{\lambda}(x).
\]

Fourier series on compact Lie groups and symmetric spaces are examples. In
particular, the eigenfunctions of the Laplace operator on the torus are
trigonometric functions and eigenfunction expansions are classical Fourier
series. Similarly, eigenfunctions of the Laplace operator on the surface of a
sphere are homogeneous harmonic polynomials and eigenfunction expansions are
spherical harmonic expansions. In the setting of manifolds, an analog of
trigonometric polynomials is given by finite linear combinations of
eigenfunctions $\sum_{\lambda}c_{\lambda}\varphi_{\lambda}(x)$. Indeed it can
be shown that there is a close relation between approximation by functions of
exponential type and by eigenfunctions. See for example \cite{Colzani-Masiero}%
. The following generalizes Theorem \ref{1} and Corollary \ref{2}.

\begin{theorem}
\label{16}{\ Given }$\alpha>0${, there exists }$\beta>0${\ such that, for
every domain }$\Omega${\ in }$\mathcal{M}${\ and }$R>0${, there exist linear
combinations }$A(x)${\ and }$B(x)${\ of eigenfunctions with eigenvalues at
most }$R^{2}${, satisfying\ }
\[
A(x)\leq\chi_{\Omega}(x)\leq B(x),\;\;\;\left|  B(x)-A(x)\right|
\leqslant\beta\left(  1+R\operatorname*{dist}\left(  x,\partial\Omega\right)
\right)  ^{-\alpha}.
\]

\end{theorem}

It is likely that a slightly more precise result holds, with a rapid decay
instead of a polynomial decay $\left(  1+R\operatorname*{dist}\left(
x,\partial\Omega\right)  \right)  ^{-\alpha}$. However, the exponent $\alpha$
can be arbitrarily large and this suffices for our applications.

\begin{proof}
The proof of this theorem is similar to the proof of Theorem \ref{1} and it is
based on suitable approximations of the identity adapted to the manifold,
analogous to the convolution kernels in the Euclidean spaces:
\begin{gather*}
K_{R}(x,y)=%
%TCIMACRO{\dsum _{\lambda<R}}%
%BeginExpansion
{\displaystyle\sum_{\lambda<R}}
%EndExpansion
c_{\lambda}\varphi_{\lambda}(x)\overline{\varphi_{\lambda}(y)},\\
\left|  K_{R}(x,y)\right|  \leq c\left(  \alpha\right)  R^{d}\left(
1+R\operatorname*{dist}\left(  x,y\right)  \right)  ^{-\alpha},\\%
%TCIMACRO{\dint _{\mathcal{M}}}%
%BeginExpansion
{\displaystyle\int_{\mathcal{M}}}
%EndExpansion
K_{R}(x,y)d\mu(y)=1.
\end{gather*}

Moreover, these kernels are positive up to a negligible error. The
construction of such kernels on Lie groups and symmetric spaces is well known
and in these cases it is possible to obtain positivity. Indeed, if a kernel
has good decay and finite spectrum, then also its square has good decay and
finite spectrum and a suitable normalization has mean one. We do not know
whether positivity can be achieved in our general setting, however in the
sequel almost positivity will suffice. Given $m\left(  \xi\right)  $ as in the
proof of Theorem \ref{1} and
\[
h\left(  \left\vert \xi\right\vert \right)  =\left(  1+\left\vert
\xi\right\vert ^{2}\right)  ^{-(d+1)/2}m\ast m\left(  \xi\right)  ,
\]
define
\[
K_{R}(x,y)=%
%TCIMACRO{\dsum _{\lambda}}%
%BeginExpansion
{\displaystyle\sum_{\lambda}}
%EndExpansion
h\left(  R^{-1}\lambda\right)  \varphi_{\lambda}(x)\overline{\varphi_{\lambda
}(y)}.
\]

It is possible to prove that this kernel has an asymptotic expansion with
Euclidean main term $R^{d}K(R\operatorname*{dist}(x,y))$ and suitable control
on the remainder. Although the details are not completely trivial, the
techniques can be found in Chapter XII of \cite{Taylor}, or in
\cite{Brandolini-Colzani}. Finally, define
\begin{gather*}
H_{R}(x)=\beta\left(  1+R\operatorname*{dist}\left(  x,\partial\Omega\right)
\right)  ^{-\alpha},\\
A(x)=%
%TCIMACRO{\dint _{\mathcal{M}}}%
%BeginExpansion
{\displaystyle\int_{\mathcal{M}}}
%EndExpansion
K_{R}(x,y)\left(  \chi_{\Omega}(y)-H_{R}(y)\right)  d\mu(y),\\
B(x)=%
%TCIMACRO{\dint _{\mathcal{M}}}%
%BeginExpansion
{\displaystyle\int_{\mathcal{M}}}
%EndExpansion
K_{R}(x,y)\left(  \chi_{\Omega}(y)+H_{R}(y)\right)  d\mu(y).
\end{gather*}

Then, as in the proof of Theorem \ref{1}, it is possible to show that for some
$\beta>0$ independent of $\Omega$ and $R$ these functions satisfy the required properties.
\end{proof}

\begin{theorem}
\label{17}{\ For every sequence of points }$\left\{  x_{j}\right\}  _{j=1}%
^{m}${\ and domain }$\Omega${\ in $\mathcal{M}$} {and }$R>0${, if }$H_{R}%
(x)${\ is defined as in the proof of Theorem \ref{16}, then }
\begin{gather*}
\left\vert \mu(\Omega)-m^{-1}%
%TCIMACRO{\dsum _{j=1}^{m}}%
%BeginExpansion
{\displaystyle\sum_{j=1}^{m}}
%EndExpansion
\chi_{\Omega}(x_{j})\right\vert \\
\leq\left\vert \widehat{H}_{R}(0)\right\vert +%
%TCIMACRO{\dsum _{0<\lambda<R}}%
%BeginExpansion
{\displaystyle\sum_{0<\lambda<R}}
%EndExpansion
\left(  \left\vert \widehat{\chi}_{\Omega}(\lambda)\right\vert +\left\vert
\widehat{H}_{R}(\lambda)\right\vert \right)  \left\vert m^{-1}%
%TCIMACRO{\dsum _{j=1}^{m}}%
%BeginExpansion
{\displaystyle\sum_{j=1}^{m}}
%EndExpansion
\varphi_{\lambda}\left(  x_{j}\right)  \right\vert .
\end{gather*}

\end{theorem}

\begin{proof}
This proof is completely analogous to the one of Theorem \ref{4}.
\end{proof}

Of course, the interest of the above result arises when one is able to exhibit
point distributions with $m^{-1}\sum_{j=1}^{m}\varphi_{\lambda}\left(
x_{j}\right)  $ suitably small. Inspired by \cite{Lubotzky-Phillips-Sarnak} on
the problem of distributing points on a sphere, we now consider point
distributions generated by the action of a free group on a homogeneous space.
Let $\mathcal{G}$\ be a compact Lie group, $\mathcal{K}$\ a closed subgroup,
$\mathcal{M}=\mathcal{G}/\mathcal{K}$\ a homogeneous space of dimension
$d$\ with normalized invariant measure $\mu$. Let $\mathcal{H}$\ be a finitely
generated free subgroup in $\mathcal{G}$\ and assume that the action of
$\mathcal{H}$\ on $\mathcal{M}$\ is free. Given a positive integer $k$, let
$\left\{  \sigma_{j}\right\}  _{j=1}^{m}$ be an ordering of the elements in
$\mathcal{H}$ with length at most $k$. For every function $f(x)$\ on
$\mathcal{M}$, define
\[
Tf(x)=m^{-1}%
%TCIMACRO{\dsum _{j=1}^{m}}%
%BeginExpansion
{\displaystyle\sum_{j=1}^{m}}
%EndExpansion
f\left(  \sigma_{j}x\right)  .
\]

This operator is self adjoint with norm $1$, hence all its eigenvalues have
modulus at most $1$. Indeed, $1$ is an eigenvalue and the constants are
eigenfunctions. In the following, we shall be interested in cases where all
other non constant eigenfunctions have eigenvalues much smaller than $1$. For
this reason, define $\rho(m)$ as the supremum of the eigenvalues with non
constant eigenfunctions,
\[
\rho(m)=\sup_{T\varphi\left(  x\right)  =\nu\varphi\left(  x\right)
,\,\,\varphi\left(  x\right)  \neq1}\left\vert \nu\right\vert .
\]

Moreover, as before, define $M\left(  \delta,\Omega\right)  =\sup
_{t>0}t^{-\delta}\mu\left(  \left\{  \operatorname*{dist}\left(
x,\partial\Omega\right)  <t\right\}  \right)  $.

\begin{theorem}
{There exists a positive constant }$c${\ such that for every point }$x${\ in
}$\mathcal{M}${\ and }$R>0${,}
\[
\sup_{M\left(  \delta,\Omega\right)  <\gamma}\left|  \mu(\Omega)-m^{-1}%
%TCIMACRO{\dsum _{j=1}^{m}}%
%BeginExpansion
{\displaystyle\sum_{j=1}^{m}}
%EndExpansion
\chi_{\Omega}(\sigma_{j}x)\right|  \leq c\gamma\left(  R^{-\delta
}+R^{(d-\delta)/2}\rho(m)\right)  .
\]

\end{theorem}

\begin{proof}
Since the operators $T$ and $\Delta$ commute, they have a common orthonormal
system of eigenfunctions, $\Delta\varphi_{\lambda}(x)=\lambda^{2}%
\varphi_{\lambda}(x)$ and $T\varphi_{\lambda}(x)=T(\lambda)\varphi_{\lambda
}(x)$. The assumption in the theorem is precisely that $\left\vert
T(\lambda)\right\vert \leq\rho(m)$ if $\lambda\neq0$. Hence, by Theorem
\ref{17},
\begin{gather*}
\left\vert \mu(\Omega)-m^{-1}%
%TCIMACRO{\dsum _{j=1}^{m}}%
%BeginExpansion
{\displaystyle\sum_{j=1}^{m}}
%EndExpansion
\chi_{\Omega}(\sigma_{j}x)\right\vert \\
\leq\left\vert \widehat{H}_{R}(0)\right\vert +\rho(m)%
%TCIMACRO{\dsum _{0<\lambda<R}}%
%BeginExpansion
{\displaystyle\sum_{0<\lambda<R}}
%EndExpansion
\left(  \left\vert \widehat{\chi}_{\Omega}(\lambda)\right\vert +\left\vert
\widehat{H}_{R}(\lambda)\right\vert \right)  \left\vert \varphi_{\lambda
}\left(  p\right)  \right\vert .
\end{gather*}
Since $H_{R}\left(  x\right)  =\beta\left(  1+R\operatorname*{dist}\left(
x,\partial\Omega\right)  \right)  ^{-\alpha}$, then
\begin{gather*}
\widehat{H}_{R}(0)\leq\beta%
%TCIMACRO{\dint _{\mathcal{M}}}%
%BeginExpansion
{\displaystyle\int_{\mathcal{M}}}
%EndExpansion
\left(  1+R\operatorname*{dist}\left(  x,\partial\Omega\right)  \right)
^{-\alpha}d\mu(x)\\
\leq\beta\left(  \mu\left(  \left\{  \operatorname*{dist}\left(
x,\partial\Omega\right)  \leq R^{-1}\right\}  \right)  +%
%TCIMACRO{\dsum _{j=0}^{+\infty}}%
%BeginExpansion
{\displaystyle\sum_{j=0}^{+\infty}}
%EndExpansion
2^{-\alpha j}\mu\left(  \left\{  \operatorname*{dist}\left(  x,\partial
\Omega\right)  <2^{j+1}R^{-1}\right\}  \right)  \right) \\
\leq\beta\left(  1+2^{\delta}%
%TCIMACRO{\dsum _{j=0}^{+\infty}}%
%BeginExpansion
{\displaystyle\sum_{j=0}^{+\infty}}
%EndExpansion
2^{(\delta-\alpha)j}\right)  M\left(  \delta,\Omega\right)  R^{-\delta}.
\end{gather*}
Similarly, by Cauchy and Bessel inequalities,
\begin{gather*}%
%TCIMACRO{\dsum _{\lambda<R}}%
%BeginExpansion
{\displaystyle\sum_{\lambda<R}}
%EndExpansion
\left\vert \widehat{H}_{R}(\lambda)\right\vert \left\vert \varphi_{\lambda
}\left(  p\right)  \right\vert \leq\left\{
%TCIMACRO{\dsum _{\lambda<R}}%
%BeginExpansion
{\displaystyle\sum_{\lambda<R}}
%EndExpansion
\left\vert \widehat{H}_{R}(\lambda)\right\vert ^{2}\right\}  ^{1/2}\left\{
%TCIMACRO{\dsum _{\lambda<R}}%
%BeginExpansion
{\displaystyle\sum_{\lambda<R}}
%EndExpansion
\left\vert \varphi_{\lambda}\left(  p\right)  \right\vert ^{2}\right\}
^{1/2}\\
\leq\beta\left\{
%TCIMACRO{\dint _{\mathcal{M}}}%
%BeginExpansion
{\displaystyle\int_{\mathcal{M}}}
%EndExpansion
\left(  1+R\operatorname*{dist}\left(  x,\partial\Omega\right)  \right)
^{-2\alpha}d\mu(x)\right\}  ^{1/2}\left\{
%TCIMACRO{\dsum _{\lambda<R}}%
%BeginExpansion
{\displaystyle\sum_{\lambda<R}}
%EndExpansion
\left\vert \varphi_{\lambda}\left(  p\right)  \right\vert ^{2}\right\}
^{1/2}\\
\leq\beta\left\{  1+2^{\delta}%
%TCIMACRO{\dsum _{j=0}^{+\infty}}%
%BeginExpansion
{\displaystyle\sum_{j=0}^{+\infty}}
%EndExpansion
2^{(\delta-2\alpha)j}\right\}  ^{1/2}\left\{
%TCIMACRO{\dsum _{\lambda<R}}%
%BeginExpansion
{\displaystyle\sum_{\lambda<R}}
%EndExpansion
\left\vert \varphi_{\lambda}\left(  p\right)  \right\vert ^{2}\right\}
^{1/2}\sqrt{M\left(  \delta,\Omega\right)  }R^{-\delta/2}.
\end{gather*}
One also gets
\begin{gather*}%
%TCIMACRO{\dsum _{0<\lambda<R}}%
%BeginExpansion
{\displaystyle\sum_{0<\lambda<R}}
%EndExpansion
\left\vert \widehat{\chi}_{\Omega}(\lambda)\right\vert \left\vert
\varphi_{\lambda}\left(  p\right)  \right\vert \\
\leq\left\{
%TCIMACRO{\dsum _{0<\lambda<1}}%
%BeginExpansion
{\displaystyle\sum_{0<\lambda<1}}
%EndExpansion
\left\vert \widehat{\chi}_{\Omega}(\lambda)\right\vert ^{2}\right\}
^{1/2}\left\{
%TCIMACRO{\dsum _{0<\lambda<1}}%
%BeginExpansion
{\displaystyle\sum_{0<\lambda<1}}
%EndExpansion
\left\vert \varphi_{\lambda}\left(  p\right)  \right\vert ^{2}\right\}
^{1/2}\\
+%
%TCIMACRO{\dsum _{k=0}^{\left[  \log_{2}(R)\right]  }}%
%BeginExpansion
{\displaystyle\sum_{k=0}^{\left[  \log_{2}(R)\right]  }}
%EndExpansion
\left\{
%TCIMACRO{\dsum _{\lambda\geq2^{k}}}%
%BeginExpansion
{\displaystyle\sum_{\lambda\geq2^{k}}}
%EndExpansion
\left\vert \widehat{\chi}_{\Omega}(\lambda)\right\vert ^{2}\right\}
^{1/2}\left\{
%TCIMACRO{\dsum _{\lambda<2^{k+1}}}%
%BeginExpansion
{\displaystyle\sum_{\lambda<2^{k+1}}}
%EndExpansion
\left\vert \varphi_{\lambda}\left(  p\right)  \right\vert ^{2}\right\}
^{1/2}.
\end{gather*}

If $A_{2^{k}}(x)\leq\chi_{\Omega}(x)\leq B_{2^{k}}(x)$ are the approximating
functions in Theorem \ref{16} corresponding to $R=2^{k}$, then,
\begin{gather*}%
%TCIMACRO{\dsum _{\lambda\geq2^{k}}}%
%BeginExpansion
{\displaystyle\sum_{\lambda\geq2^{k}}}
%EndExpansion
\left\vert \widehat{\chi}_{\Omega}(\lambda)\right\vert ^{2}\leq%
%TCIMACRO{\dint _{\mathcal{M}}}%
%BeginExpansion
{\displaystyle\int_{\mathcal{M}}}
%EndExpansion
\left\vert \chi_{\Omega}(x)-A_{2^{k}}(x)\right\vert ^{2}d\mu(x)\\
\leq\beta^{2}%
%TCIMACRO{\dint _{\mathcal{M}}}%
%BeginExpansion
{\displaystyle\int_{\mathcal{M}}}
%EndExpansion
\left(  1+2^{k}\operatorname*{dist}\left(  x,\partial\Omega\right)  \right)
^{-2\alpha}d\mu(x)\\
\leq\beta^{2}\left(  1+2^{\delta}%
%TCIMACRO{\dsum _{j=0}^{+\infty}}%
%BeginExpansion
{\displaystyle\sum_{j=0}^{+\infty}}
%EndExpansion
2^{(\delta-2\alpha)j}\right)  M\left(  \delta,\Omega\right)  2^{-\delta k}.
\end{gather*}

By classical bounds on the spectral function of an elliptic operator, see e.g.
Theorem 17.5.3 in \cite{Hormander},
\[%
%TCIMACRO{\dsum _{\lambda<R}}%
%BeginExpansion
{\displaystyle\sum_{\lambda<R}}
%EndExpansion
\left\vert \varphi_{\lambda}\left(  p\right)  \right\vert ^{2}\leq cR^{d}.
\]

Hence, if $\alpha$ is large enough,
\begin{gather*}
\left\vert \mu(\Omega)-m^{-1}%
%TCIMACRO{\dsum _{j=1}^{m}}%
%BeginExpansion
{\displaystyle\sum_{j=1}^{m}}
%EndExpansion
\chi_{\Omega}(\sigma_{j}p)\right\vert \\
\leq cM\left(  \delta,\Omega\right)  R^{-\delta}+c\sqrt{M\left(  \delta
,\Omega\right)  }R^{\left(  d-\delta\right)  /2}\rho(m).
\end{gather*}

Finally, observe that $\sqrt{M\left(  \delta,\Omega\right)  }\leq cM\left(
\delta,\Omega\right)  $, since $M\left(  \delta,\Omega\right)  $ is bounded
below as one sees putting $t=1$ in the definition of this constant.
\end{proof}

The following corollary is Theorem \ref{sphereintro} in the Introduction, and
it has been proved in \cite{Lubotzky-Phillips-Sarnak} in the case of spherical caps.

\begin{corollary}
\label{sphere} {If }$\mathcal{M}=SO(3)/SO(2)${\ is the two dimensional sphere,
if }$\mathcal{H}${\ is the free group generated by rotations of angles
}$\arccos(-3/5)${\ around orthogonal axes, then there exists a constant }$c$
{such that,} {if} $k${ is an integer and }$\left\{  \sigma_{j}\right\}
_{j=1}^{m}${ is an ordering of the elements in }$\mathcal{H}${ with length at
most }$k$,{ then for every }$x$,
\[
\left\vert \mu(\Omega)-m^{-1}%
%TCIMACRO{\dsum _{j=1}^{m}}%
%BeginExpansion
{\displaystyle\sum_{j=1}^{m}}
%EndExpansion
\chi_{\Omega}(\sigma_{j}x)\right\vert \leq cM\left(  \delta,\Omega\right)
m^{-\delta/(2+\delta)}\log^{2\delta/(2+\delta)}(m).
\]

\end{corollary}

\begin{proof}
The eigenvalues of the operator $T$ satisfy the Ramanujan bounds
\[
\rho(m)=\sup_{T\varphi\left(  x\right)  =\nu\varphi\left(  x\right)
,\,\,\varphi\left(  x\right)  \neq1}\left\vert \nu\right\vert \leq
cm^{-1/2}\log(m).
\]
Hence, choosing $R=m^{1/(d+\delta)}\log^{-2/(d+\delta)}(m)$ in the above
theorem,
\[
\inf_{R>0}\left\{  R^{-\delta}+R^{\left(  d-\delta\right)  /2}m^{-1/2}%
\log(m)\right\}  \leq cm^{-\delta/(d+\delta)}\log^{2\delta/(d+\delta)}(m).
\]

\end{proof}

%Bibliographies can be prepared with BibTeX using amsplain,
%amsalpha, or (for "historical" overviews) natbib style.

%Insert the bibliography data here.

%\[
%\text{\textbf{References}}%
%\]

\bigskip

\noindent\textsc{Dipartimento di Matematica e Applicazioni, Edificio U5}

\noindent\textsc{Universit\`{a} di Milano-Bicocca}

\noindent\textsc{Via R.Cozzi 53}

\noindent\textsc{20125 Milano, Italy}

\noindent\texttt{leonardo.colzani@unimib.it}

\bigskip

\noindent\textsc{Dipartimento di Ingegneria dell'Informazione e Metodi
Mate\-matici}

\noindent\textsc{Universit\`{a} di Bergamo, Viale Marconi 5}

\noindent\textsc{24044 Dalmine, Bergamo, Italy}

\noindent\texttt{giacomo.gigante@unibg.it}

\bigskip

\noindent\textsc{Dipartimento di Statistica, Edificio U7}

\noindent\textsc{Universit\`{a} di Milano-Bicocca}

\noindent\textsc{Via Bicocca degli Arcimboldi 8}

\noindent\textsc{20126 Milano, Italy}

\noindent\texttt{giancarlo.travaglini@unimib.it}

\end{document}